%% file: HI-Vfil.tex
\documentclass[11pt]{amsart}
\usepackage{amsmath,amsfonts,mathrsfs}

\newif\ifdraft
\draftfalse
\input{macros.tex}

\newcommand{\I}{\mathcal{I}}

\def\ZZ{{\mathbf Z}}

\def\CC{{\mathbf C}}

\def\QQ{{\mathbf Q}}

\newtheorem*{thmA'}{Theorem~A$^\prime$}

\begin{document}

\vspace{\baselineskip}

\title{Hodge ideals for $\QQ$-divisors, $V$-filtration, and minimal exponent}

\author[M. Musta\c{t}\u{a}]{Mircea~Musta\c{t}\u{a}}
\address{Department of Mathematics, University of Michigan,
Ann Arbor, MI 48109, USA}
\email{{\tt mmustata@umich.edu}}

\author[M.~Popa]{Mihnea~Popa}
\address{Department of Mathematics, Northwestern University, 
2033 Sheridan Road, Evanston, IL
60208, USA} \email{{\tt mpopa@math.northwestern.edu}}

\thanks{MM was partially supported by NSF grant DMS-1701622 and a Simons Fellowship; MP was partially supported by NSF grant DMS-1700819.}


\begin{abstract}
We compute the Hodge ideals of $\QQ$-divisors in terms of the $V$-filtration induced  by a local defining equation, inspired by a result of Saito in the reduced case. We deduce basic properties of Hodge ideals in this generality, and relate them to Bernstein-Sato polynomials. As a consequence of our study we establish general properties of the minimal exponent, a refined version of the log canonical threshold, 
and bound it in terms of discrepancies on log resolutions, addressing a question of Lichtin and Koll\'ar. 
\end{abstract}

\maketitle

\makeatletter

\section{Introduction}

This paper establishes a connection between the Hodge ideals of a $\QQ$-divisor, as defined in \cite{MP3}, and the $V$-filtration
along an appropriately chosen hypersurface. It is inspired by Saito's \cite{Saito-MLCT}, which explained such a connection expressed in terms of the microlocal $V$-filtration, in the case of the Hodge ideals of reduced divisors studied in \cite{MP1}.
In the $\QQ$-divisor case,  this relationship turns out to be crucial towards establishing some of the most basic properties 
of Hodge ideals, as well as of certain roots of Bernstein-Sato polynomials.

Let $X$ be a smooth complex variety, and $D$ an effective $\QQ$-divisor on $X$. Such a divisor can be written locally as $D = \alpha H$, where $\alpha \in \QQ$, and $H = {\rm div} (f)$ is the divisor of a regular function, and it is this set-up that we focus on in what follows. To this data, by a standard construction one associates the left $\Dmod_X$-module
$$\Mmod (f^\beta) : = \shO_X(*H)f^{\beta},$$ 
a free $\shO_X(*H)$-module of rank $1$ with generator the symbol $f^{\beta}$, where $\beta = 1 - \alpha$ (the $\Dmod$-module action is recalled in 
\S\ref{set-up}). In \cite{MP3} we observe that it carries a natural filtration $F_p \Mmod (f^\beta)$, with $p \ge 0$, which makes 
it a filtered direct summand in a $\Dmod$-module underlying a mixed Hodge module. Moreover, we show that this can be written in the form
$$F_p \Mmod (f^\beta) = I_p (D) \otimes \shO_X (pZ + H) f^\beta,$$
with $Z = H_{\rm red}$ the support of $H$, where $I_p(D)$ are coherent sheaves of ideals on $X$ called the Hodge ideals of $D$. Note that here and throughout the paper we make a slight abuse of notation, identifying the right-hand side with its image via the canonical injection into  $\shO_X(*H)f^{\beta}$.

The  ideal $I_0 (D)$ is identified in \emph{loc. cit.} with the multiplier ideal $\I \big((1-\epsilon)D\big)$ associated to the $\QQ$-divisor $(1- \epsilon)D$ with $0 < \epsilon \ll 1$, which measures the failure of the pair $(X, D)$ to be log canonical.
On the other hand, when $D$ is integral Budur and Saito \cite{Budur-Saito} have shown the identification 
$$\I \big((1-\epsilon)D\big) = V^1 \shO_X,$$
where $V^\bullet \shO_X$ denotes the $V$-filtration induced on $\shO_X$. This is defined as
$V^\bullet \shO_X = V^\bullet \iota_+ \shO_X \cap \shO_X$, where $V^\bullet \iota_+ \shO_X$ is the  
Kashiwara-Malgrange $V$-filtration on the graph embedding (via a local equation of $H$) of $\shO_X$; 
see \S\ref{scn:Vfil}. Consequently we have $I_0 (D) = V^1 \shO_X$.

When $D$ is a reduced divisor (corresponding in the notation above to $\beta = 0$ and  $D = H = Z$), Saito showed in \cite{Saito-MLCT} that a relationship of this type continues to hold in a weaker sense even for $p\ge 1$, namely 
$$I_p (D) =  \widetilde{V}^{p +1} \shO_X  \,\,\,\,\,\,{\rm mod} ~f,$$
or in other words $I_p (D)  + (f) =  \widetilde{V}^{p +1} \shO_X  + (f)$, where this time 
$\widetilde{V}^{\bullet} \shO_X$
denotes the \emph{microlocal $V$-filtration} induced on $\shO_X$ by $f$, defined in \cite{Saito_microlocal}. Examples show that in general this identification does not hold without modding out by $f$; even so however, it is significant for a number of reasons. Most importantly, it establishes a connection between Hodge ideals and the Bernstein-Sato polynomial of $f$. Moreover, to establish the triviality of the ideals on the two sides, it suffices to check it mod $f$. 

In this paper we show that similar statements hold for arbitrary $\QQ$-divisors. More precisely, we fully compute the Hodge ideals in terms of 
the (usual) $V$-filtration on $\iota_+\shO_X$. This strengthens Saito's result above even in the reduced case. In order to state our results, let us recall first that without loss of generality it suffices to focus on the case $\lceil D \rceil = Z$. Indeed, 
the $\QQ$-divisor $B = D + Z - \lceil D \rceil $ satisfies this condition, while
according to \cite[Lemma 4.4]{MP3}, we have
$$I_p (D) = I_p (B)\otimes \shO_X (Z - \lceil D \rceil ).\footnote{In other words, up to multiplication by locally principal ideals, the Hodge ideals depend only the coefficients of $D$ modulo integers, just as in the well-known case of multiplier ideals (which corresponds to $p =0$).}$$
It will also be convenient to express things equivalently in terms of a slightly different ideal, defined by the formula
$$F_p \Mmod (f^\beta) = I''_p (D) \otimes \shO_X \big((p+1)H\big) f^\beta,\footnote{The notation $I'_p (D)$ is reserved for a 
different object in \cite{MP3}.}$$
or in other words satisfying 
$$I''_p (D) = I_p (D)\otimes \shO_X (pZ - pH).$$
The $V$-filtration will come into play via the following construction: for each $p \ge 0$, we consider the coherent sheaf of ideals in $\shO_X$ given by
$$\widetilde{I}_p(D) = \left\{ v\in\shO_X~|~ \exists~ v_0,\ldots,v_{p-1},v_p=v \in \shO_X {\rm ~with~} 
\sum_{i=0}^pv_i\partial_t^i\delta\in V^{\alpha}\iota_+\shO_X \right\}.$$
For $0 < \alpha \le 1$, this is just another way of expressing Saito's microlocal $V$-filtration mentioned above; specifically, one has 
$$\widetilde{I}_p (D) = \widetilde{V}^{p +\alpha} \shO_X.$$ 
It will also be convenient to make use of the polynomials
$$Q_i (X) : = X (X+1) \cdots (X + i - 1)\in \ZZ[X].$$

With these definitions and reductions, our main result can be phrased as follows:

\begin{intro-theorem}\label{general_description}
In the set-up above, for every positive rational number $\alpha$ such that $D= \alpha H$ satisfies $\lceil D\rceil=Z$, and for every $p\ge 0$, we have
$$I''_p (D)=\left\{\sum_{j=0}^pQ_j(\alpha)f^{p-j}v_j ~| ~\sum_{j=0}^pv_j\partial_t^j\delta\in V^{\alpha}\iota_+\shO_X\right\}.$$
In particular, we have
$$I''_p (D)+(f)=\widetilde{I}_p(D)+(f).$$
\end{intro-theorem}

Note that in the case where $D = \alpha Z$, with $Z$ a reduced effective divisor, we have 
$I''_p (D)=I_p(D)$ for all $p\geq 0$. In this case we have the following variant of the theorem above, 
where we place no restrictions on the positive rational number $\alpha$.

\begin{thmA'}\label{general_description_variant}
If $Z$ is a reduced, effective divisor on $X$, defined by the global equation $f\in\shO_X(X)$, then
for every positive rational number $\alpha$, and every $p\geq 0$, if $D = \alpha Z$ we have
\begin{equation}\label{eq_general_description_variant}
I_p(D)=\left\{\sum_{j=0}^pQ_j(\alpha)f^{p-j}v_j ~|~ \sum_{j=0}^pv_j\partial_t^j\delta\in V^{\alpha}\iota_+\shO_X\right\}.
\end{equation}
In particular, we have
$$I_p(D)+(f)=\widetilde{I}_p(D)+(f).$$
\end{thmA'}

The proofs of these theorems, as well as various intermediate results, occupy \S\ref{one_way} and \S\ref{other_way}.
Some of the arguments follow \cite{Saito-MLCT}, and rely on the regular and quasi-unipotent property of 
filtered $\Dmod$-modules underlying mixed Hodge modules. For the full calculation of Hodge ideals in terms of the 
$V$-filtration however,  further techniques need to be developed as well. One technical point,  of independent interest, is a calculation  of the $V$-filtration on the (graph embedding of the) twisted $\Dmod$-modules $\Mmod (f^\beta)$  in terms of the more tractable $V$-filtration on $\iota_+\shO_X$; for the statement see Proposition \ref{correspondence}.  

\noindent
{\bf Remark.} The microlocal $V$-filtration has been computed explicitly in various cases by Saito; for instance, it is computed for a large class of quasi-homogeneous isolated singularities in \cite[Proposition (2.2.4)] {Saito-MLCT}.  In \cite{Saito-HF} the Hodge filtration itself is computed combinatorially for all such singularities, and this is extended to the case of $\QQ$-divisors in \cite{Zhang}.
This leads to an explicit calculation of Hodge ideals for isolated quasi-homogeneous singularities; see \emph{loc. cit.} for examples.

In the $\QQ$-divisor case, the results above allow us to deduce some basic properties of Hodge ideals that do not follow from the methods of \cite{MP3}. We collect some of these, treated individually and discussed in detail in  \S\ref{basic}, in the following:

\begin{intro-corollary}\label{intro_properties}
Let $D = \alpha Z$, where $Z$ is a reduced divisor and $\alpha \in \QQ_{> 0}$. Then the following hold:
\begin{enumerate}
\item $I_p (D) + \shO_X(-Z) \subseteq I_{p-1} (D) + \shO_X(-Z)$ for all $p$.
\item If $(X, D)$ is $(p-1)$-log canonical,\footnote{This means that $I_0 (D) = \cdots = I_{p-1} (D) = \shO_X$. In particular it requires $\alpha \le 1$.} 
then $I_{p+1} (D) \subseteq I_p (D)= \widetilde{I}_p (D)$.
\item Fixing $p$, there exists a finite set of rational numbers $0 = c_0 < c_1 < \cdots < c_s < c_{s+1} = 1$ such that 
for each $0 \le i \le s$ and each $\alpha \in (c_i, c_{i+1}]$ we have 
$$I_p (\alpha Z)\cdot\shO_Z = I_p (c_{i+1} Z)\cdot\shO_Z = {\rm constant}$$
and such that 
$$I_p (c_{i+1} Z)\cdot\shO_Z\subsetneq I_p (c_i Z)\cdot\shO_Z.$$
\end{enumerate}
\end{intro-corollary}

The last statement gives a picture analogous to that of jumping coefficients of multiplier ideals \cite[Lemma ~9.3.21]{Lazarsfeld}. If $f$ is a local equation of $Z$, the set of $c_i$ is a subset of the set of 
jumping numbers for the $V$-filtration on $\iota_+\shO_X$  associated to $f$. An example in \S\ref{basic} shows that the statement fails if we work directly with $I_p (\alpha Z)$ 
as opposed to $I_p (\alpha Z)\cdot\shO_Z$.

There are interesting applications,  obtained in \S\ref{refined_lct} by combining the above results  with the birational study of Hodge ideals in \cite{MP3}, concerning the Bernstein-Sato polynomial $b_Z(s)$ of $Z$. Assuming $Z\neq 0$, the polynomial $(s+1)$ divides $b_Z(s)$. Following \cite{Saito-MLCT}, we denote by 
$\widetilde\alpha_Z$ the negative of the largest root of $b_Z(s)/(s+1)$ (with the convention that this is $\infty$ if $b_Z(s)=s+1$). 
This invariant is called the \emph{minimal exponent} of $Z$, and is a refined version of the log canonical threshold of the pair $(X, Z)$, which is equal to ${\rm min}\{\widetilde\alpha_Z, 1\}$; see \S\ref{refined_lct} for a discussion. First, since by Theorem \ref{general_description_variant}$^\prime$ 
we have that $I_p (D)$ is trivial if and only if $\widetilde{I}_p (D)$ is so, results of Saito on the microlocal $V$-filtration will allow us to conclude:

\begin{intro-corollary}\label{mlc_intro}
If $Z\neq 0$ is a reduced effective divisor on the smooth variety $X$ and $\alpha\in (0,1]$ is a rational number, then
$$I_p (\alpha Z) = \shO_X \iff p\leq\widetilde{\alpha}_Z - \alpha.$$
\end{intro-corollary}

Now given a  log resolution $\mu \colon Y\to X$ of the pair $(X,Z)$, assumed to be an isomorphism over $X\smallsetminus Z$, if $F_1,\ldots,F_m$ are the irreducible components of its exceptional locus and $\widetilde{Z}$ is the strict transform of $Z$, assumed to be smooth, we write
$$\mu^* Z =\widetilde{Z}+\sum_{i=1}^ma_iF_i\quad\text{and}\quad K_{Y/X}=\sum_{i=1}^mb_iF_i,$$
where $K_{Y/X}$ is the relative canonical divisor. Denoting
$$\gamma : = \underset{i=1, \ldots, m}{\rm min} \left\{\frac{b_i +1}{a_i}\right\},$$
it is well known that the log canonical threshold of $(X, Z)$ is also equal to $\min\{\gamma,1\}$. Such a precise interpretation in terms of log resolutions is however not known for other roots of the Bernstein-Sato polynomial, and 
Lichtin \cite[Remark~2,~p.303]{Lichtin} posed the natural question whether $\widetilde\alpha_Z = \gamma$. As noted by Koll\'ar \cite[Remark~10.8]{Kollar}, in general the answer is negative, since $\gamma$ in fact depends on the choice of log resolution. Nevertheless:

\begin{intro-corollary}\label{lichtin_inequality}
With the notation above, we always have $\widetilde\alpha_Z \ge \gamma$. 
\end{intro-corollary}

The reason is that the triviality of $I_p (D)$ is related on one hand to $\widetilde\alpha_Z$ by Corollary \ref{mlc_intro}, and on the other hand to $\gamma$ by \cite[Proposition 11.2]{MP3}. 
It is worth noting that, although the statement is about the reduced divisor $Z$, the proof uses crucially the theory of Hodge ideals for $\QQ$-divisors of the form $D = \alpha Z$.

Using further properties of Hodge ideals of ${\mathbf Q}$-divisors proved in \cite{MP3}, we deduce some general properties of the minimal exponent $\widetilde{\alpha}_D$ for any effective divisor $D$, extending important features of the log canonical threshold. In order to formulate the result, it is convenient to use a local version of this refined log canonical threshold, denoted
$\widetilde{\alpha}_{D,x}$, for $x\in D$ (see \S\ref{refined_lct} for the precise definition).

\begin{intro-theorem}\label{properties_refined_lct}
Let $X$ be a smooth $n$-dimensional  complex variety, and $D$ an effective divisor on $X$.
\begin{enumerate}
\item If $Y$ is a smooth subvariety of $X$ such that $Y\not\subseteq D$, then for every
$x\in D\cap Y$, we have
$$\widetilde{\alpha}_{D\vert_Y,x}\leq \widetilde{\alpha}_{D,x}.$$
\item Consider a smooth morphism $\pi\colon X\to T$, together with a section $s\colon T\to X$ such that 
$s(T)\subseteq D$.  If $D$ does not contain any fiber of $\pi$, so that
for every $t\in T$ the divisor $D_t=D\vert_{\pi^{-1}(t)}$ is defined, then the function
$$T\ni t\to \widetilde{\alpha}_{D_t,s(t)}$$
is lower semicontinuous.
\item For every $x\in X$, if $m={\rm mult}_x(D)\geq 2$, then
$$ \frac{n-r-1}{m} \le \widetilde{\alpha}_{D,x}\le \frac{n}{m},$$
where $r$ is the dimension of the singular locus of the projectivized tangent cone
${\mathbf P}(C_xD)$ of $D$ at $x$ (with the convention that $r=-1$ if ${\mathbf P}(C_xD)$ is smooth).
\end{enumerate}
\end{intro-theorem}

When $D$ has an isolated singularity at $x$ and $Y$ is a general hyperplane section through $x$, the inequality in (1) was proved in \cite[Th\'eor\`eme~1]{Loeser} (in fact, in this case the inequality is strict).
The semicontinuity property in (2) was proved when every $D_t$ has an isolated singularity at $s(t)$ in \cite[Theorem~2.11]{Steenbrink}, where it was deduced from more general semicontinuity properties of the spectrum. We stress that in (2) we do not assume that the restriction of the support of $D$ 
to the fibers of $\pi$ is reduced, as in the semicontinuity theorem \cite[Theorem 14.1]{MP3} (which we do use).

Yet more properties analogous to those of log canonical thresholds follow by combining Theorem \ref{properties_refined_lct} with a Thom-Sebastiani-type theorem due to Saito; see Proposition \ref{further_properties} for the concrete statement. We ask in Question \ref{q_ACC} whether the analogue of the ACC property for log canonical thresholds holds for minimal exponents as well.

Finally, going back to the general relationship between Hodge ideals and the Bernstein-Sato polynomial, in Proposition \ref{roots_b}  we give an extension of the fact that the negatives of the jumping coefficients of multiplier ideals in the interval $(0,1]$ are roots of the Bernstein-Sato polynomial, see \cite[Theorem~B]{ELSV}. Namely, under a suitable log-canonicity hypothesis, the jumping coefficients of higher Hodge ideals in the same interval, in the sense of Corollary \ref{intro_properties} (3), lead to further such roots. This follows quickly from results proved in the final two sections of the paper.

\noindent
{\bf Acknowledgements.}
 We are grateful to Morihiko Saito for comments and suggestions that helped improve a previous version of this paper. We would also like to thank Nero Budur and Mingyi Zhang for a few useful discussions, and a referee for several corrections.

\section{Main results}

\subsection{The set-up.}\label{set-up}
Let $X$ be a smooth complex algebraic variety and $H$ an effective divisor on $X$.
We assume that $H$ is defined by a global regular function $f\in\shO_X(X)$. We denote by $\shO_X(*H)$ the sheaf of
rational functions on $X$ with poles along $H$, that is, 
$$\shO_X(*H)=\bigcup_{m\geq 0}\shO_X(mH).$$
Given a rational number $\gamma$, we consider 
the $\Dmod_X$-module 
$$\Mmod(f^{\gamma}):=\shO_X(*H)f^{\gamma}.$$ 
This is a free $\shO_X(*H)$-module of rank $1$,
with generator the symbol $f^{\gamma}$, on which a derivation $D$ of $\shO_X$ acts by
$$D (w f^{\gamma})=\big(D(w) + w\frac{\gamma\cdot D(f)}{f}\big)f^{\gamma}.$$
We will keep the notation $\shO_X(*H)$ for $\Mmod(f^0)$.
The $\Dmod_X$-modules $\Mmod(f^{\gamma})$ are regular holonomic, with quasi-unipotent monodromy.
In fact, they are filtered direct summands of $\Dmod$-modules underlying mixed Hodge modules, see \cite[\S2]{MP3}.

Note that if $\gamma_1-\gamma_2=d$ is an integer, then we have a canonical isomorphism
of $\Dmod_X$-modules
$$\Mmod(f^{\gamma_1})\overset{\simeq}{\longrightarrow} \Mmod(f^{\gamma_2}),\quad gf^{\gamma_1}\mapsto (gf^d)f^{\gamma_2}.$$
We will be particularly interested in the $\Dmod$-modules $\Mmod(f^{\beta})$ as in the Introduction, which via the isomorphism above can also be identified with $\Mmod(f^{-\alpha})$ with $\alpha = 1 - \beta$.

\subsection{The $V$-filtrations corresponding to $\shO_X(*H)$ and $\Mmod(f^{\beta})$}\label{scn:Vfil}
We begin by reviewing some basic facts about $V$-filtrations.
Let 
$$\iota\colon X\hookrightarrow X\times\CC, \,\,\,\,\, x\mapsto \big(x,f(x)\big)$$ 
be the closed embedding given by the graph of $f$.
For a $\Dmod_X$-module $\Mmod$, we consider
the $\Dmod$-module theoretic direct image 
$$\iota_+\Mmod : = \Mmod \otimes_\CC \CC[\partial_t],$$
see for instance \cite[Example~1.3.5]{HTT}. 
This is a $\Dmod_{X\times\CC}$-module that can be described as follows. First, if $\Mmod=\shO_X$, then 
$$\iota_+\shO_X \simeq \shO_X[t]_{f-t}/\shO_X[t],$$ 
with the obvious $\Dmod_X$-module structure. If $\delta$ denotes the class of $\frac{1}{f-t}$ in $\iota_+\shO_X$, it is straightforward to see that every element in $\iota_+\shO_X$ can be written uniquely as
$$\sum_{j\geq 0}h_j \partial_t^j\delta,$$
with $h_j\in\shO_X$, only finitely many of these being nonzero.
Note that by definition we have $t\delta=f\delta$.

Given an arbitrary $\Dmod_X$-module $\Mmod$, we have
$$\iota_+\Mmod\simeq \Mmod\otimes_{\shO_X}\iota_+\shO_X=\bigoplus_{j\geq 0}\Mmod\otimes_{\shO_X}\shO_X\partial_t^j\delta,$$
which in particular shows the connection with the original definition above.
With this description, multiplication by $t$ is given by
\begin{equation}\label{formula_mult_t}
t(m\otimes\partial_t^j\delta)=fm\otimes\partial_t^j\delta-jm\otimes\partial_t^{j-1}\delta
\end{equation}
and the action of a derivation $D\in {\rm Der}_{\CC}(\shO_X)$ is given by
$$D(m\otimes\partial_t^j\delta)=D(m)\otimes \partial_t^j\delta-D(f)m\otimes\partial_t^{j+1}\delta.$$
In particular, every element in $\iota_+\Mmod(f^{\beta})$ can be written uniquely as a finite sum
$$\sum_{j\geq 0}h_jf^{\beta}\otimes\partial_t^j\delta,\quad\text{with}\quad h_j\in\shO_X(*H).$$ 
For a derivation $D$ as above, we have
$$D(hf^{\beta}\otimes\partial_t^j\delta)=D(h)f^{\beta}\otimes\partial_t^j \delta +\beta \frac{D(f)h}{f} f^{\beta}\otimes\partial_t^j\delta-D(f)hf^{\beta}\otimes\partial_t^{j+1}\delta.$$

For a $\Dmod_X$-module $\Mmod$, we call a \emph{$V$-filtration} on $\iota_+\Mmod$ 
a rational filtration $(V^{\gamma}=V^{\gamma}\iota_+\Mmod)_{\gamma\in\QQ}$
that is exhaustive, decreasing, discrete, and left continuous,\footnote{More precisely, the filtration has the property
that there is a positive integer $\ell$ such that $V^{\gamma}$ takes constant value in each interval $\left(\frac{i}{\ell},\frac{i+1}{\ell}\right]$, for all $i\in \ZZ$.}
such that the following conditions are satisfied:
\begin{enumerate}
\item[i)] Each $V^{\gamma}$ is a coherent module over $\Dmod_X[t,\partial_tt]$.
\item[ii)] For every $\gamma\in\QQ$, we have an inclusion
$$t\cdot V^{\gamma}\subseteq V^{\gamma+1},$$
with equality if $\gamma>0$.
\item[iii)] For every $\gamma\in\QQ$, we have
$$\partial_t\cdot V^{\gamma}\subseteq V^{\gamma-1}.$$
\item[iv)] For every $\gamma\in\QQ$, if we put $V^{>\gamma}=\bigcup_{\gamma'>\gamma}V^{\gamma'}$, then
$\partial_tt-\gamma$ acts nilpotently on 
$${\rm Gr}_V^{\gamma}:=V^{\gamma}/V^{>\gamma}.$$
\end{enumerate}

It is easy to see that there exists at most one $V$-filtration (see for example \cite[Lemme~3.1.2]{Saito-MHP}).
The existence of the $V$-filtration for $\Mmod=\shO_X$ and $\Mmod=\shO_X(*H)$ was proved by Malgrange 
\cite{Malgrange}; the case of an arbitrary holonomic $\Mmod$ is due to Kashiwara  \cite{Kashiwara3}.
We note that this original $V$-filtration was indexed by integers; the indexing by $\QQ$ (in the case of a regular holonomic $\Mmod$, with quasi-unipotent monodromy) was introduced by Saito \cite{Saito-GM}.

\begin{remark}\label{rmk_filtration_O_X}
Every element in the cokernel of the inclusion $\shO_X\hookrightarrow\shO_X(*H)$ is annihilated by some power of $f$. This implies that the canonical inclusion
$\iota_+\shO_X\hookrightarrow\iota_+\shO_X(*H)$ induces equalities
$$V^{\gamma}\iota_+\shO_X=V^{\gamma}\iota_+\shO_X(*H)\quad\text{for all}\quad\gamma>0$$
(see \cite[Lemme~3.1.7]{Saito-MHP}).
\end{remark}

In what follows, it will be convenient to also have a different description of $\iota_+\Mmod$ under 
a minor extra assumption on $\Mmod$. From now on we assume that multiplication by $f$ is bijective on $\Mmod$
(in other words, $\Mmod$ has a natural structure of $\shO_X(*H)$-module). Note that this applies, in particular, if
$\Mmod=\Mmod(f^{\beta})$.

\begin{remark}\label{bij_f_mult}
Our hypothesis on $\Mmod$ implies that multiplication by $t$ is bijective on $\iota_+\Mmod$. 
Indeed, if we consider on $\iota_+\Mmod$ the filtration given by 
$$G_p=G_p\iota_+\Mmod:=\bigoplus_{j=0}^p\Mmod\otimes_{\shO_X}\shO_X\partial_t^j\delta,$$
then multiplication by $t$ preserves the filtration; moreover, 
it follows from (\ref{formula_mult_t}) that
for every $p\geq 0$, via the obvious isomorphism
$G_p/G_{p-1}\simeq\Mmod$, multiplication by $t$ gets identified with multiplication by $f$. We thus obtain by induction on $p$ 
the fact that multiplication
by $t$ on $G_p$ is an isomorphism.
\end{remark}

\begin{remark}\label{V_filt_bij_f_mult}
If we assume that $\Mmod$ is as above and there is a $V$-filtration on $\iota_+\Mmod$, then
$$t\cdot V^{\alpha}=V^{\alpha+1}\quad\text{for all}\quad \alpha\in\QQ,$$
where for simplicity we denote $V^{\alpha}= V^{\alpha} \iota_+\Mmod$.
Indeed, the inclusion ``$\subseteq$", as well as the reverse inclusion for $\alpha>0$, follow from
general properties of the $V$-filtration. Moreover, the induced map
$${\rm Gr}(t)\colon {\rm Gr}^{\delta}_V \to  {\rm Gr}^{\delta+1}_V$$
is an isomorphism if $\delta\neq 0$. 

Suppose now that $\alpha\leq 0$ and $u=tw\in V^{\alpha+1}$. Let $\delta\ll 0$ be such that $w\in V^{\delta}$.
If $\delta\geq\alpha$, then we are done. On the other hand, if $\delta<\alpha$, then $\delta\neq 0$ since $\alpha\leq 0$; since $tw\in V^{>\delta+1}$, we conclude that 
$w\in V^{>\delta}$. 
After repeating this argument finitely many times, we obtain $w\in V^{\alpha}$. 
\end{remark}

Let $\Dmod\langle t,s\rangle$ be the subsheaf of $\Dmod_{X\times\CC}$ generated by $\Dmod_X$, $t$, and $s=-\partial_tt$. 
Note that $t$ and $s$ satisfy $st=t(s-1)$ and more generally
\begin{equation}\label{commutator_s_and_t}
P(s)t=t P(s-1)\quad\text{for all}\quad P\in\CC[s].
\end{equation}
We also consider the localization $\Dmod_X\langle t,t^{-1},s\rangle=\Dmod_X\langle t,t^{-1},\partial_t\rangle$ of $\Dmod\langle t,s\rangle$.
(This is the push-forward of the sheaf of differential operators from $X\times\CC^*$ to $X\times\CC$.)
Note that in this ring we have $\partial_t=-st^{-1}$ and from (\ref{commutator_s_and_t}) we obtain
\begin{equation}\label{commutator_s_and_t_v2}
t^{-1}P(s)=P(s-1)t^{-1}\quad\text{for all}\quad P\in\CC[s].
\end{equation}
A $\Dmod_X \langle t,t^{-1},s\rangle$-module is simply a $\Dmod_{X\times\CC}$-module on which $t$ acts bijectively.

We consider the $\Dmod_X\langle t,t^{-1},s\rangle$-module $\Mmod[s]f^s$ defined as follows.
As an $\shO_X$-module, we have an isomorphism
$$\Mmod\otimes_{\shO_X}\shO_X[s]\simeq \Mmod[s]f^s,\quad u\otimes s^j\to us^jf^s.$$
The symbol $f^s$ motivates the $\Dmod_X$-action: a derivation $D$ in ${\rm Der}_{\CC}(\shO_X)$ acts by
$$D( us^jf^s)=\big(D(u) s^j+ \frac{D(f)}{f}us^{j+1}\big)f^s.$$
The action of $s$ on $\Mmod[s]f^s$ is the obvious one, while the action of $t$ is given by the automorphism
``$s\to s+1$", that is
$$us^jf^s\to fu(s+1)^jf^s.$$

\begin{remark}
In light of Remarks~\ref{bij_f_mult} and \ref{V_filt_bij_f_mult}, if $\Mmod$ is a $\Dmod_X$-module on which multiplication by $f$ is bijective, a $V$-filtration on $\iota_+\Mmod$
can be characterized as an exhaustive, decreasing, discrete, left continuous, rational filtration $(V^{\gamma}=V^{\gamma}\iota_+\Mmod)_{\gamma\in\QQ}$, that satisfies the following conditions:
\begin{enumerate}
\item[i')] Each $V^{\gamma}$ is a coherent module over $\Dmod_X\langle t, t^{-1}, s\rangle$.
\item[ii')] For every $\gamma\in\QQ$, we have
$$t\cdot V^{\gamma}=V^{\gamma+1}.$$
\item[iii)] For every $\gamma\in\QQ$, we have
$$\partial_t\cdot V^{\gamma}\subseteq V^{\gamma-1}.$$
\item[iv)] For every $\gamma\in\QQ$, 
the operator $s+\gamma$ acts nilpotently on ${\rm Gr}_V^{\gamma}$.
\end{enumerate}
\end{remark}

The next proposition contains the promised  
description of $\iota_+\Mmod$. While the result is well known (in fact, in the case
$\Mmod=\shO_X(*H)$, this has already been noticed in \cite{Malgrange}), we sketch the proof
since we will need the explicit description of the isomorphism.
For every $i\geq 0$, we put
$$Q_i(x)=i!\cdot {{x+i-1}\choose i}:=\prod_{j=0}^{i-1}(x+j)\in\ZZ[x]$$
(with the convention $Q_0=1$).

\begin{proposition}\label{alternative_description}
If $\Mmod$ is a $\Dmod_X$-module on which $f$ acts bijectively, then we have an isomorphism
of $\Dmod_X\langle t,t^{-1},s\rangle$-modules
\begin{equation}\label{eq_alternative_description}
\Mmod[s]f^s\simeq \iota_+\Mmod,\quad us^jf^s\to u\otimes(-\partial_tt)^j\delta.
\end{equation}
The inverse isomorphism is given by
\begin{equation}\label{eq2_alternative_description}
u\otimes\partial_t^j\delta\to \frac{u}{f^j}Q_j(-s)f^s.
\end{equation}
\end{proposition}

\begin{proof}
It is straightforward to check that the map in (\ref{eq_alternative_description}) is $\Dmod_X\langle t,t^{-1},s\rangle$-linear.
In order to see that it is an isomorphism, consider on $\Mmod[s]f^s$ and $\iota_+\Mmod$ the filtrations given by
$$G_p\Mmod[s]f^s=\bigoplus_{i=0}^p\Mmod s^jf^s\quad\text{and}\quad G_p\iota_+\Mmod=\bigoplus_{j=0}^p\Mmod\otimes_{\shO_X}\shO_X\partial_t^j\delta.$$
Note that the map (\ref{eq_alternative_description}) preserves the filtrations. Moreover, we have canonical isomorphisms
$$G_p\Mmod[s]f^s/G_{p-1}\Mmod[s]f^s\simeq\Mmod\simeq G_p\iota_+\Mmod/G_{p-1}\iota_+\Mmod$$
such that the map induced by (\ref{eq_alternative_description}) is given by multiplication with $(-1)^pf^p$.
Since this is an isomorphism, we conclude by induction on $p$ that each induced map $G_p\Mmod[s]f^s\to G_p\iota_+\Mmod$
is an isomorphism. 

The formula for the inverse isomorphism follows if we show that in $\iota_+\shO_X(*H)$ we have
$$Q_j(\partial_tt)\delta=f^j\partial_t^j\delta\quad\text{for all}\quad j\geq 0.$$
We argue by induction on $j$, the case $j=0$ being obvious. Assuming the formula for some $j$, we apply 
(\ref{commutator_s_and_t_v2}) and the fact that $\partial_t=-st^{-1}$ to write
$$f^{j+1}\partial_t^{j+1}\delta=f\partial_tQ_j(-s)\delta=-fst^{-1}Q_j(-s)\delta=f(-s)Q_j(-s-1)t^{-1}\delta=Q_{j+1}(-s)\delta.$$
This completes the proof of the proposition.
\end{proof}

\medskip

We now come to the main result of this section, relating the
$V$-filtrations on $\iota_+\Mmod(f^{\beta})$ and $\iota_+\shO_X(*H)$ as follows.

\begin{proposition}\label{correspondence}
For every $\beta\in\QQ$, we have an isomorphism of $\Dmod_X\langle t,t^{-1},s\rangle$-modules
$$\Phi\colon \iota_+\Mmod(f^{\beta})\to \iota_+\shO_X(*H),$$
where on the right-hand side $\Dmod_X\langle t,t^{-1},s\rangle$ acts via the automorphism 
$\Dmod_X\langle t,t^{-1},s\rangle\to \Dmod_X\langle t,t^{-1},s\rangle$ that maps $s$ to $s-\beta$ and is the identity
on $\Dmod_X$ and on $t$. The isomorphism $\Phi$ is given by 
$$\Phi\left(\sum_{i=0}^p h_if^{\beta} \otimes\partial_t^i\delta \right)=\sum_{i=0}^pg_i\otimes\partial_t^i\delta,$$
where 
\begin{equation}\label{eq1_correspondence}
\sum_{i=0}^p\frac{h_i}{f^i}Q_i(-s+\beta)=\sum_{i=0}^p\frac{g_i}{f^i}Q_i(-s)\quad\text{in}\quad \shO_X(*H)[s].
\end{equation}
The isomorphism $\Phi$ translates the $V$-filtration by $-\beta$, in the sense that
\begin{equation}\label{eq3_correspondence}
\Phi\big(V^{\gamma}\iota_+\Mmod(f^{\beta})\big)=V^{\gamma-\beta}\iota_+\shO_X(*H)\quad\text{for every}\quad \gamma\in\QQ.
\end{equation}
\end{proposition}

\begin{proof}
The isomorphism is more transparently described via the identifications provided by Proposition~\ref{alternative_description}, which gives isomorphisms
$$\varphi_1\colon \iota_+\Mmod(f^{\beta})\to \Mmod(f^{\beta})[s]f^s,\quad\varphi_1(uf^{\beta}\otimes\partial_t^i\delta)=\frac{u}{f^i}f^{\beta}Q_i(-s)f^s\quad\text{and}$$
$$\varphi_2\colon \iota_+\shO_X(*H)\to \shO_X(*H)[s]f^s,\quad\varphi_2(u\otimes\partial_t^i\delta)=\frac{u}{f^i}Q_i(-s)f^s.$$
It is then straightforward to check that if we define
$$\Phi'\colon \Mmod(f^{\beta})[s]f^s\to \shO_X(*H)[s]f^s,\quad (uf^{\beta})P(s)f^s\to uP(s-\beta)f^s,$$
this is an isomorphism of $\Dmod_X\langle t,t^{-1},s\rangle$-modules, where the action on the right-hand side is via the automorphism of 
$\Dmod_X\langle t,t^{-1},s\rangle$ described in the statement of the proposition. 
If we take $\Phi=\varphi_2^{-1}\circ\Phi'\circ\varphi_1$, we obtain an isomorphism that satisfies (\ref{eq1_correspondence}).
The fact that $\Phi$ translates the $V$-filtration by $-\beta$ is a consequence of the uniqueness of the $V$-filtration and of the fact that our
automorphism of $\Dmod_X\langle t,t^{-1},s\rangle$ is the identity on $\Dmod_X$ and $t$ and maps $s$ to $s-\beta$. 
\end{proof}

\begin{remark}\label{attrib_Saito}
In a first version of this paper, we showed that the isomorphism $\Phi$ translates the $V$-filtration by $-\beta$ using the description of this filtration
in terms of Bernstein-Sato polynomials due to Sabbah \cite{Sabbah} (see Proposition~\ref{sabbah} below). 
The argument above, based on the uniqueness of the $V$-filtration, was pointed out to us by M.~Saito.
\end{remark}

We next give a more explicit description of the transformation $\Phi$.

\begin{proposition}\label{precise_formula}
If $\Phi\colon \iota_+\Mmod(f^{\beta})\to \iota_+\shO_X(*H)$ is the map in Proposition~\ref{correspondence},
and if $\Phi(u)=v$, where 
$$u=\sum_{i=0}^ph_if^{\beta}\otimes\partial_t^i\delta\quad\text{and}\quad v=\sum_{i=0}^pg_i\otimes\partial_t^j\delta,$$
then
\begin{equation}\label{eq1_precise_formula}
\frac{h_i}{f^i}=\sum_{j=i}^p{j\choose i}Q_{j-i}(-\beta)\frac{g_j}{f^j}.
\end{equation}
\end{proposition}

\begin{proof}
Letting $x=-s$ and $y=-\beta$ in Lemma \ref{eq2_precise_formula} in the Appendix, and using the definition of $\Phi$, 
we get
$$\sum_{i=0}^p\frac{h_i}{f^i}Q_i(-s+\beta)=\sum_{j=0}^p\frac{g_j}{f^j}Q_j(-s)$$
$$=\sum_{j=0}^p\sum_{i=0}^j\frac{g_j}{f^j}{j\choose i}Q_{j-i}(-\beta)Q_i(-s+\beta)$$
$$=\sum_{i=0}^p\left(\sum_{j=i}^p\frac{g_j}{f^j}{j\choose i} Q_{j-i}(-\beta)\right)Q_i(-s+\beta).$$
Since the polynomials $Q_i(-s+\beta)$, with $0\leq i\leq p$, are linearly independent over $\QQ$, the equality between 
the first and the last expressions above gives  (\ref{eq1_precise_formula}). 
\end{proof}

\begin{remark}
Proposition \ref{precise_formula} is used below in the proof of Theorem \ref{general_description}. It was noted more recently in \cite[\S2.4]{JKSY} that there is a proof of the theorem along the same lines as here, which however avoids the use of this precise formula.
\end{remark}

\subsection{From the Hodge ideals of $D = \alpha H$ to $V^{\alpha}\iota_+\shO_X$}\label{one_way}
We now begin the study of the Hodge filtration on $\Mmod(f^{\beta})$. 
Recall that we are considering a $\QQ$-divisor $D = \alpha H$, with $\alpha \in \QQ_{> 0}$, and $H$ defined 
by $f \in \shO_X (X)$. We denote $Z=H_{\rm red}$, and assume from now on that $\lceil D \rceil =Z$.

Recall from the introduction that for every 
nonnegative integer $p$, there is an ideal sheaf  $I''_p (D)$ such that 
$$F_p \Mmod(f^{\beta}) =I''_p (D)\otimes \shO_X\big((p+1)H\big)f^{\beta}.$$
This ideal is related to the $p$-th Hodge ideal of $D$ by the formula
$$I''_p(D )=I_p(D)\cdot\shO_X(pZ-pH).$$
In particular, we see that $I''_0(D)=I_0(D)$.

\begin{definition}
For every $p \ge 0$, we define the subsheaf $\widetilde{I}_p(D)$ of $\shO_X$ by 
$$\widetilde{I}_p(D) = \left\{ v\in\shO_X~|~ \exists~ v_0,\ldots,v_{p-1},v_p=v \in \shO_X {\rm ~with~} 
\sum_{i=0}^pv_i\partial_t^i\delta\in V^{\alpha}\iota_+\shO_X \right\}.$$
Since $V^{\alpha}\iota_+\shO_X$ is an $\shO_X$-module, it follows that $\widetilde{I}_p(D)$ is a (coherent) ideal 
in $\shO_X$. As mentioned in the Introduction, when $0 < \alpha \le 1$ this is another way of expressing Saito's microlocal $V$-filtration
\cite{Saito_microlocal} induced on $\shO_X$.

We note that we have made an abuse of notation here: both ideals $I''(D)$ and $\widetilde{I}_p(D)$ 
depend on the choice of $H$, and not just on the ${\mathbf Q}$-divisor $D = \alpha H$. 
However, in what follows $H$ will be fixed, and we hope that this will not lead to any confusion.
\end{definition}

\begin{proposition}\label{first_inclusion}
For every nonnegative integer $p$ and every $g\in I''_p (D)$, 
there is 
$$v=\sum_{j=0}^pv_j\partial_t^j\delta\in V^{\alpha}\iota_+\shO_X$$
such that
\begin{equation}\label{eq1_first_inclusion}
g=\sum_{j=0}^pQ_j(\alpha)f^{p-j}v_j.
\end{equation}
In particular, we have
$$I''_p (D)\subseteq \widetilde{I}_p(D)+(f).$$
\end{proposition}

\medskip

Before giving the proof, recall that the Hodge filtration on $\Mmod(f^{\beta})$ induces a Hodge filtration on 
$\iota_+\Mmod(f^{\beta})$,  given by
$$F_p\iota_+\Mmod(f^{\beta})=\bigoplus_{j=0}^pF_{p-j}\Mmod(f^{\beta})\otimes\partial_t^j\delta, \,\,\,\,\,\,{\rm for ~all} \,\,\,\,p \ge 0.\footnote{In the literature, the right-hand side is sometimes taken to define $F_{p+1} \iota_+\Mmod(f^{\beta})$ rather than $F_p$. We find it somewhat more convenient notationally to use this convention here.}$$
This, just as with all the  filtered $\Dmod$-modules we consider here, satisfies the following special 
property.

\begin{lemma}[{\cite[3.2.3]{Saito-MHP}}]\label{lem_V_filtration}
Let $t$ be a nonzero function on the smooth variety $Y$, defining a smooth divisor $H$. If $(M,F)$ is a filtered $\Dmod_Y$-module
with no $t$-torsion, and which carries
a $V$-filtration with respect to $t$ that is compatible with the $F$-filtration in the sense of \cite[3.2]{Saito-MHP}, such that the induced 
morphism
$${\rm Gr}(t)\colon \big({\rm Gr}_V^0(M),F\big)\to \big({\rm Gr}_V^1(M),F\big)$$
is strict, then 
\begin{equation}\label{eq_lem_V_filtration}
F_pM=\sum_{i\geq 0}\partial_t^i(V^0M\cap j_*j^*F_{p-i}M)\quad\text{for all}\quad p,
\end{equation}
where $j\colon Y\smallsetminus H\hookrightarrow Y$ is the inclusion.
\end{lemma}

\begin{remark}\label{rmk_lem_V_filtration}
The conclusion of the lemma applies in particular when $(M,F)$ is a direct summand of a filtered $\Dmod_Y$-module $(N, F)$ that underlies a mixed Hodge module (and hence is regular and quasi-unipotent, so it satisfies \cite[3.2]{Saito-MHP}), and such that $N$ has no $t$-torsion and 
$${\rm Gr}(t)\colon \big({\rm Gr}_V^0(N),F\big)\to \big({\rm Gr}_V^1(N),F\big)$$
is a filtered isomorphism. We may therefore apply it
to the filtered $\Dmod_{X\times\CC}$-module $\iota_+\Mmod(f^{\beta})$. Indeed, according to 
\cite[Lemma 2.11]{MP3},  $\Mmod(f^{\beta})$ is a filtered direct summand in a $\Dmod$-module on $X$ of the form
$j_+ (Q, F)$, where $j\colon U \hookrightarrow X$ is the natural inclusion of $U = X \smallsetminus Z$, and 
$(Q, F)$ is the filtered $\Dmod$-module underlying a mixed Hodge module on $U$; hence $\big(\iota_+\Mmod(f^{\beta}), F\big)$
is a summand in $(j \times {\rm id_\CC})_+  (\iota_U)_+ (Q, F)$, where $\iota_U$ is the graph embedding corresponding to $f\vert_U$. But 
filtered $\Dmod$-modules such as the latter satisfy the properties above, by the general construction of direct images of Hodge modules via open 
embeddings in \cite[Proposition~2.8]{Saito-MHM} (cf. also \cite[Proposition~4.2]{Saito-B}).
\end{remark}

We can now prove the main result of the section.

\begin{proof}[Proof of Proposition~\ref{first_inclusion}]
The argument is similar to that in \cite{Saito-MLCT}, which treats the case when $D$ is reduced (i.e. $H$ is reduced 
and $\alpha=1$). In what follows we may, and will assume, that $X$ is affine. 

Let $g\in I''_p (D)$. It follows from the definition of $I''_p (D)$ that  we have
$$\frac{g}{f^{p+1}}f^{\beta}\otimes\delta\in F_p \iota_+\Mmod(f^{\beta}).$$
Using Remark~\ref{rmk_lem_V_filtration}, 
we may apply Lemma~\ref{lem_V_filtration} for the $\Dmod_{X\times\CC}$-module $\iota_+\Mmod(f^{\beta})$,
hence we can write
\begin{equation}\label{eq2_first_inclusion}
\frac{g}{f^{p+1}}f^{\beta}\otimes\delta=\sum_{i=0}^p\partial_t^iu^{(i)},
\end{equation}
with $u^{(i)}\in V^0 \iota_+\Mmod(f^{\beta})\cap j_*j^* F_{p-i}\iota_+\Mmod(f^{\beta})$ for all $i$. 
If we write
$$u^{(0)}=\sum_{i=0}^pu_i^{(0)}f^{\beta}\otimes\partial_t^i\delta,$$
then it follows from (\ref{eq2_first_inclusion}) that 
$$
g=f^{p+1}u^{(0)}_0.
$$
Note now that since $u^{(0)}\in V^0 \iota_+\Mmod(f^{\beta})$, we have
$$tu^{(0)}\in V^1\iota_+\Mmod(f^{\beta}),$$
and by the definition of the action of $t$, we can write
$$tu^{(0)}=\sum_{i=0}^pfu_i^{(0)}f^{\beta}\otimes\partial_t^{i}\delta -\sum_{i=1}^piu^{(0)}_i f^{\beta}\otimes\partial_t^{i-1}\delta.$$

We now use the transformation $\Phi$ in Proposition~\ref{correspondence} to deduce that
$$\Phi(tu^{(0)}) \in V^{\alpha}\iota_+\shO_X(*H)=V^{\alpha}\iota_+\shO_X,$$
where we use the fact that $V^{\gamma}\iota_+\shO_X(*H)=V^{\gamma}\iota_+\shO_X$ for every $\gamma>0$, by Remark~\ref{rmk_filtration_O_X}.
For every $j\geq i$, let 
$$a_{i,j} :={j\choose i}Q_{j-i}(-\beta).$$ 
It follows from Proposition~\ref{precise_formula} that 
there are $v_0,\ldots,v_p\in\shO_X(X)$ such that
$$fu^{(0)}_i-(i+1)u^{(0)}_{i+1}=\sum_{j=i}^pa_{i,j}\frac{v_j}{f^{j-i}},$$
with the convention that $u^{(0)}_{p+1}=0$.
Therefore we have
$$g=f^{p+1}u^{(0)}_0=\sum_{i=0}^p\big(i! f^{p+1-i}u_i^{(0)}-(i+1)! f^{p-i}u_{i+1}^{(0)}\big)
=\sum_{i=0}^p\sum_{j=i}^pi! a_{i,j}f^{p-j}v_j$$
$$=\sum_{j=0}^p f^{p-j}v_j\sum_{i=0}^j i! {j\choose i} Q_{j-i}(-\beta)=\sum_{j=0}^pQ_j(\alpha)f^{p-j}v_j,$$
where the last equality follows from Lemma~\ref{formula_alternating_sum}.
We thus have (\ref{eq1_first_inclusion}). The last assertion in the statement is clear, 
since $v_p\in \widetilde{I}_p(\alpha H)$ and $Q_p(\alpha)\neq 0$.
\end{proof}

\subsection{From $V^{\alpha}\iota_+\shO_X$ to the Hodge ideals of $D = \alpha H$}\label{other_way}
Keeping the notation of \S\ref{one_way},  the following is the main result of this section:

\begin{proposition}\label{second_inclusion}
For every nonnegative integer $p$, if $v=\sum_{j=0}^pv_j\partial_t^j\delta \in V^{\alpha}\iota_+\shO_X$, then 
\begin{equation}\label{eq1_second_inclusion}
\sum_{j=i}^p{j\choose i}Q_{j-i}(\alpha)f^{p-j}v_j\in I''_{p-i}(D)\quad\text{for}\quad 0\leq i\leq p.
\end{equation}
In particular, we have
$$\widetilde{I}_p(D)\subseteq I''_p (D)+(f).$$
\end{proposition}

\begin{remark}\label{first_ones_by_induction}
We will prove the proposition by induction on $p$. Note that if we know it for all $q<p$, 
then we know the statements in
(\ref{eq1_second_inclusion}) for $1\leq i\leq p$.
Indeed, since $v\in V^{\alpha}\iota_+\shO_X$, we also have
$$V^{\alpha}\iota_+\shO_X\ni (f-t)v=\sum_{j=1}^pjv_j\partial_t^{j-1}\delta.$$
Iterating this, we conclude that for every $1\leq i\leq p$, we have
$$V^{\alpha}\iota_+\shO_X\ni (f-t)^iv=\sum_{j=i}^p\frac{j!}{(j-i)!}v_j\partial_t^{j-i}\delta.$$
Applying the inductive hypothesis for $(f-t)^iv$, we conclude that we have
$$\sum_{j=i}^p\frac{j!}{(j-i)!}Q_{j-i}(\alpha)f^{p-j} v_j\in I''_{p-i}(D),$$
as claimed.
\end{remark}

\begin{remark}\label{formula_w}
Let us explain the significance of the sums on the left-hand side of (\ref{eq1_second_inclusion}). Suppose that $v=\sum_{i=0}^pv_i\partial_t^i\delta
\in V^{\alpha}\iota_+\shO_X$ and 
$$u=\sum_{i=0}^pu_if^{\beta}\otimes \partial_t^i\delta \in V^1\iota_+\Mmod(f^{\beta})$$ is such that $\Phi(u)=v$. Note that multiplication by $t$ is bijective on 
$\iota_+\Mmod(f^{\beta})$,
and let $w$ be such that $u=tw$. If we write $w=\sum_{i=0}^pw_if^{\beta}\otimes\partial_t^i\delta$, then
\begin{equation}\label{sums_w}
w_i=\sum_{j=i}^p {j\choose i}Q_{j-i}(\alpha)\frac{v_j}{f^{j-i+1}},
\end{equation}
hence
$$f^{p-i +1} w_i$$
are precisely the sums on the left hand side of (\ref{eq1_second_inclusion}).

To check ($\ref{sums_w}$), if we denote by $w'_i$ the right hand side of the formula,  it is enough to show that
$$fw'_i-(i+1)w'_{i+1}=u_i\quad\text{for}\quad 0\leq i\leq p-1,$$
and $fw'_p=u_p$. Since $w'_p=\frac{1}{f}v_p=\frac{1}{f}u_p$, the last equality is clear.
Note now that
$$w'_i=\sum_{j=i}^p\frac{v_j}{f^{j-i+1}}{j\choose i}\cdot\prod_{k=1}^{j-i}(k-\beta).$$
It follows that if $0\leq i\leq p-1$, then
$$fw'_i-(i+1)w'_{i+1}=\sum_{j=i}^p\frac{v_j}{f^{j-i}}{j\choose i}\cdot\prod_{k=1}^{j-i}(k-\beta)-\sum_{j=i+1}^p\frac{v_j}{f^{j-i}}(i+1){j\choose {i+1}}\cdot\prod_{k=1}^{j-i-1}(k-\beta)$$
$$=v_i+\sum_{j=i+1}^p\frac{v_j}{f^{j-i}}\cdot\prod_{k=1}^{j-i-1}(k-\beta)\cdot\left((j-i-\beta){j\choose i}-(i+1){j\choose i+1}\right)$$
$$=
v_i-\sum_{j=i+1}^p\frac{v_j}{f^{j-i}}\cdot {j\choose i}\beta\cdot  \prod_{k=1}^{j-i-1}(k-\beta),$$
where the last equality follows from the fact that 
$$(j-i){j\choose i}=(i+1){j\choose {i+1}}.$$
Using Proposition~\ref{precise_formula}, we thus conclude that
$$fw'_i-(i+1)w'_{i+1}=v_i+\sum_{j=i+1}^p {j\choose i}Q_{j-i}(-\beta)\frac{v_j}{f^{j-i}}=u_i.$$
\end{remark}

\begin{remark}\label{inclusion_in_I0}
It is shown in \cite[Proposition 9.1]{MP3} that, 
if ${\mathcal I}(\gamma H)$ denotes  the multiplier ideal of the $\QQ$-divisor $\gamma H$,  we have
$${\mathcal I}\big((\alpha-\epsilon)H\big)=I_0(\alpha H)=I''_0(\alpha H)$$
for $0<\epsilon\ll 1$.
We refer to \cite[Chapter~9]{Lazarsfeld} for the definition and basic properties of multiplier ideals.
In particular, for every $\alpha\leq 1$, we have
$$(f)={\mathcal I}(H)\subseteq {\mathcal I}\big((\alpha-\epsilon)H\big)=I_0(\alpha H).$$
\end{remark}

\medskip

The following property of the Hodge filtration on $\iota_+\Mmod(f^{\beta})$ is probably
well known to the experts, but we include a proof for the benefit of the reader. 

\begin{lemma}\label{prop_V_filtration}
For every $p\in\ZZ$ and every $\gamma\geq 0$, we have
$$t\cdot V^{\gamma}F_p\iota_+\Mmod(f^{\beta})=V^{\gamma+1}F_p\iota_+\Mmod(f^{\beta}).$$
\end{lemma}

\begin{proof}
In order to simplify the notation, we write $V^{\gamma}$ for $V^{\gamma}\iota_+\Mmod(f^{\beta})$ and $F_p$ for $F_p\iota_+\Mmod(f^{\beta})$.
Note first that since multiplication by $f$ is bijective on $\Mmod(f^{\beta})$, it follows that multiplication by $t$ on $\iota_+\Mmod(f^{\beta})$
is bijective and
$$t\cdot V^{\gamma} =V^{\gamma+1}\quad\text{for all}\quad\gamma\in\QQ.$$
(See Remarks~\ref{bij_f_mult} and \ref{V_filt_bij_f_mult}.)

The inclusion ``$\subseteq$" is clear, and when $\gamma>0$ the equality follows from the compatibility of
the $F$ and $V$ filtrations, see \cite[\S 3.2]{Saito-MHP}. (We use again the fact that 
$\iota_+\Mmod(f^{\beta})$ is a filtered direct summand of a mixed Hodge module.) Suppose now that 
$\gamma=0$. We also know that
$${\rm Gr}(t)\colon ({\rm Gr}_V^0, F) \to  ({\rm Gr}_V^1, F)$$
is a filtered isomorphism (see Remark~\ref{rmk_lem_V_filtration}). The statement follows then from the Five Lemma applied to the filtered commutative diagram
$$
\begin{tikzcd}
0 \rar & (V^{>0}, F) \rar\dar{\cdot t} & (V^0, F) \rar\dar{\cdot t} & ({\rm Gr}_V^0, F) \rar\dar{\cdot {\rm Gr}(t)} & 0\\
0 \rar & (V^{>1}, F) \rar & (V^1, F) \rar& ({\rm Gr}_V^1, F) \rar & 0.
\end{tikzcd}
$$
\end{proof}

We can now prove the main result of this section.

\begin{proof}[Proof of Proposition~\ref{second_inclusion}]
We argue by induction on $p$. The case $p=0$ is known: 
if $v_0\otimes\delta\in V^{\alpha}\iota_+\shO_X$, then it follows from
\cite{Budur-Saito} that $v_0\in {\mathcal I}\big((\alpha-\epsilon)H\big)$. 
On the other hand, 
we have
$${\mathcal I}\big((\alpha-\epsilon)H\big)=I_0(D)=I''_0(D)$$
by Remark~\ref{inclusion_in_I0}.
We thus obtain the statement of the proposition for $p=0$.

Suppose now that the statement holds for all $q\leq p$, and let us prove it for $p+1$. 
Let 
$$v=\sum_{i=0}^{p+1}v_i\partial_t^i\delta\in V^{\alpha}\iota_+\shO_X,$$
 and let 
 $$u=\sum_{i=0}^{p+1}u_if^{\beta}\otimes\partial_t^i\delta
\in V^1\iota_+\Mmod(f^{\beta})$$
 such that $\Phi(u)=v$. We also consider the unique $w=\sum_{i=0}^{p+1}w_if^{\beta}\otimes\partial_t^i\delta$ such that $tw=u$.
 Note that $w\in V^0\iota_+\Mmod(f^{\beta})$ by Lemma~\ref{prop_V_filtration}.
We need to show that 
\begin{equation}\label{eq2_second_inclusion}
\sum_{j=i}^{p+1}{j\choose i}Q_{j-i}(\alpha)f^{p+1-j}v_j\in I''_{p+1-i}(D)\quad\text{for}\quad 0\leq i\leq p+1.
\end{equation}
We have seen in Remark~\ref{first_ones_by_induction} that  for $1\leq i\leq p+1$ the assertion follows from the induction hypothesis, hence we only need to prove it for $i=0$.

On the other hand, it follows from Remark~\ref{formula_w} that 
$$w_i=\sum_{j=i}^{p+1} {j\choose i}Q_{j-i}(\alpha)\frac{v_j}{f^{j-i+1}}.$$
The statement in (\ref{eq2_second_inclusion}) is thus equivalent to
$$
f^{p+2-i}w_i\in I''_{p+1-i}(D)\quad\text{for}\quad 0\leq i\leq p+1,
$$
that is,
\begin{equation}\label{eq3_second_inclusion}
w_i f^{\beta}\in F_{p+1-i}\Mmod(f^{\beta})\quad\text{for}\quad 0\leq i\leq p+1.
\end{equation}
Therefore, equivalently, we know the statement in (\ref{eq3_second_inclusion}) for $1\leq i\leq p+1$, 
and we need to show it for $i = 0$.
We also record the fact that, due to the way the $F$-filtration is defined on $\iota_+\Mmod(f^{\beta})$,
the conditions in (\ref{eq3_second_inclusion}) are equivalent to the statement $w\in F_{p+1}\iota_+\Mmod(f^{\beta})$. 

Note now that we have 
$$u_i=fw_i-(i+1)w_{i+1}\quad\text{for}\quad 0\leq i\leq p+1$$
(with the convention $w_{p+2}=0$). Therefore
$$u_i f^{\beta} \in F_{p+1-i}\Mmod(f^{\beta})+F_{p-i}\Mmod(f^{\beta})\subseteq F_{p+1-i}\Mmod(f^{\beta})$$
for $1\leq i\leq p+1$. 
Furthermore, since
$$u_0=fw_0-w_1$$
and
$$w_1f^{\beta}\in F_p\Mmod(f^{\beta}),$$
we conclude that $u_0f^{\beta}\in F_{p+1}\Mmod(f^{\beta})$ 
(which, given what we already know inductively, is equivalent to $u\in F_{p+1}\iota_+\Mmod(f^{\beta})$)
if and only if $fw_0f^{\beta}\in F_{p+1}\Mmod(f^{\beta})$ (which again, given what we already know,
is equivalent to $fw\in F_{p+1}\iota_+\Mmod(f^{\beta})$).
Using Lemma~\ref{prop_V_filtration} we thus conclude that
$$w_0f^{\beta}\in F_{p+1}\Mmod(f^{\beta}) \iff fw_0f^{\beta}\in F_{p+1}\Mmod(f^{\beta}).$$
We can now apply the same argument with $v$ replaced by $fv,\ldots,f^{p+1}v$ to conclude that
$$w_0f^{\beta}\in F_{p+1}\Mmod(f^{\beta}) \iff  f^{p+2}w_0f^{\beta}\in F_{p+1}\Mmod(f^{\beta}).$$
However, it follows from Remark~\ref{formula_w} that
$$f^{p+2}w_0=\sum_{j=0}^{p+1}Q_j(\alpha)v_jf^{p+1-j}$$
is a section of $\shO_X$. Now using Remark~\ref{inclusion_in_I0} we see that
$$\shO_Xf^{\beta}\subseteq F_0\Mmod(f^{\beta})
\subseteq F_{p+1}\Mmod(f^{\beta}),$$
and putting everything together we conclude that $w_0f^{\beta}\in F_{p+1}\Mmod(f^{\beta})$. As we have seen, this completes the proof of ($\ref{eq3_second_inclusion}$), and thus of the proposition.
\end{proof}

Theorem \ref{general_description} now follows by combining Propositions~\ref{first_inclusion} and \ref{second_inclusion}.
Let us explain how we can remove the condition on $\alpha$ when $H=Z$.

\begin{proof}[Proof of Theorem~\ref{general_description_variant}$^\prime$]
It is of course enough to prove only the first assertion of the theorem.
For $\alpha\in (0,1]$, this follows from Theorem~\ref{general_description}. Therefore it suffices to show that if we know 
(\ref{eq_general_description_variant}) for $\alpha$, then we also know it for $\alpha+1$. Let us temporarily denote the 
right-hand side of  (\ref{eq_general_description_variant}) by $\sigma(\alpha Z)$. 

Note that $I_p\big((\alpha+1)Z\big)=f\cdot I_p(\alpha Z)$ by  \cite[Lemma 4.4]{MP3}), hence it is enough to show that we also have
$\sigma\big((\alpha+1)Z\big)=f\cdot \sigma(\alpha Z)$. 
By the definition of the $V$-filtration, since $\alpha>0$ we have
$$V^{\alpha+1}\iota_+\shO_X=t\cdot V^{\alpha}\iota_+\shO_X.$$
It follows that, given $v=\sum_{j=0}^pv_j\partial_t^j\delta\in V^{\alpha+1}\iota_+\shO_X$, we can find
$w=\sum_{j=0}^pw_j\partial_t^j\delta\in V^{\alpha}\iota_+\shO_X$ such that $v=tw$.
This means that 
$$v_j=fw_j-(j+1)w_{j+1}\quad\text{for}\quad 0\leq j\leq p,$$
with the convention that $w_{p+1}=0$. Let us denote by $h$ and $g$ the elements of $\sigma_p(\alpha Z)$ 
and $\sigma_p\big((\alpha+1)Z\big)$
corresponding to $w$ and $v$, respectively. We thus have
$$g=\sum_{j=0}^pQ_j(\alpha+1)f^{p-j}v_j=\sum_{j=0}^pQ_j(\alpha+1)f^{p-j}(fw_j-(j+1)w_{j+1})$$
$$=\sum_{j=0}^pf^{p-j+1}w_j\big(Q_j(\alpha+1)-j\cdot Q_{j-1}(\alpha+1)\big)=fh,$$
where the last equality follows from the fact that
$$Q_j(\alpha+1)-j\cdot Q_{j-1}(\alpha+1)=\prod_{i=1}^j(\alpha+i)-j\cdot\prod_{i=1}^{j-1}(\alpha+i)=Q_j(\alpha).$$
The equality $g=fh$ implies that $\sigma\big((\alpha+1)Z)=f\cdot \sigma(\alpha Z)$, and thus completes the proof of the theorem.
\end{proof}

We conclude with a few remarks regarding the statements of the main theorems.

\begin{remark}\label{nonreduced}
If we write $f = f_1^{m_1} \cdots f_r^{m_r}$, where $f_i$ correspond to the irreducible components of $H$ and 
$m_i \ge 1$, then $Z$ is given by the equation $g = f_1\cdots f_r$, and so  
$$I''_p (D) = I_p(D) \cdot f_1^{p(m_1- 1)} \cdots f_r^{p(m_r -1)}.$$
Thus when $p \ge 2$ and $m_i \ge 2$ for all $i$, we have $I''_p (D) \subseteq (f)$, and so the only content of the last statement in Theorem \ref{general_description} is that $\widetilde{I}_p(D) \subseteq (f)$.
\end{remark}

\begin{remark}
The last assertion in Theorem~\ref{general_description_variant}$^\prime$ is only interesting for $\alpha\leq 1$, 
since for $\alpha>1$ both sides are equal to $(f)$.
\end{remark}

\begin{remark}\label{microlocal_V_filtration}
Saito introduced and studied in \cite{Saito_microlocal} a \emph{microlocal $V$-filtration}. This induces a filtration on 
$\shO_X$ denoted 
$(\widetilde{V}^{\gamma}\shO_X)_{\gamma\in\QQ}$. Using the definition of this filtration, when $H$ is reduced one can reformulate the last assertion in Theorem~\ref{general_description} 
as saying that
$$I_p(\alpha H)\cdot \shO_H=\widetilde{V}^{p+\alpha}\shO_X\cdot\shO_H$$
for all $p\geq 0$. As mentioned in the Introduction, when $\alpha=1$, this was proved in \cite{Saito-MLCT}.
\end{remark}

\begin{remark}\label{alternative}
In the setting of Theorem~\ref{general_description}, 
the fact that the $F_p\Mmod(f^{\beta})=I''_p (D)\otimes\shO_X\big((p+1)H\big)f^{\beta}$ give a filtration on $\Mmod(f^{\beta})$ compatible with the order
filtration on $\Dmod_X$ is equivalent to the following properties:
\begin{enumerate}
\item[i)] Each $I''_p (D)$ is an $\shO_X$-module.
\item[ii)] We have $f\cdot I''_p (D)\subseteq I''_{p+1}(D)$ for every $p\geq 0$.
\item[iii)] For every $D\in {\rm Der}_{\CC}(\shO_X)$ and every $h\in I''_p (D)$, we have
$$f\cdot D(h)+\big(\beta-(p+1)\big)h\cdot D(f)\in I''_{p+1}(D).$$
\end{enumerate}
One can easily check that these properties can also be deduced from the formula in Theorem~\ref{general_description}
and the general properties of the $V$-filtration.
\end{remark}

\section{Consequences}

\subsection{Basic properties of Hodge ideals}\label{basic}
From now on we consider the case $H=Z$, that is $D = \alpha Z$, with $Z$ a reduced divisor and $\alpha$ a positive rational number. We will see that  Theorem \ref{general_description_variant}$^\prime$ implies a number of fundamental properties of Hodge ideals that cannot be easily deduced directly from the definition. 

Note that in  the statements below we do not require that $Z$ be defined by a global equation; however, the assertions immediately reduce to this case, hence in the proofs we will tacitly make this assumption, and denote by $f$ the equation defining $Z$.

\begin{corollary}\label{inclusions}
For every $p \ge 1$ we have 
$$I_p (D) + \shO_X(-Z) \subseteq I_{p-1} (D) + \shO_X(-Z).$$
\end{corollary}

\begin{proof}
If $v=\sum_{j=0}^pv_j\partial_t^j\delta\in V^{\alpha}\iota_+\shO_X$, then
$$(f-t)v=\sum_{j=1}^{p}jv_j\partial_t^{j-1}\delta\in V^{\alpha}\iota_+\shO_X.$$
We thus see that $\widetilde{I}_p(D)\subseteq \widetilde{I}_{p-1}(D)$. 
The assertion now follows from 
Theorem \ref{general_description_variant}$^\prime$.
\end{proof}

\begin{remark}
In the case $\alpha=1$ we have the stronger statement $I_{p} (D) \subseteq I_{p-1}(D)$, see 
\cite[Proposition~13.1]{MP1}. However, for $\alpha<1$ this seems likely to fail, though at the moment we 
do not have an example. It does hold when $Z$ has simple normal crossings \cite[Proposition 7.1]{MP3} and 
when $Z$ has isolated quasi-homogeneous singularities \cite{Zhang}.
\end{remark}

In what follows we will use of the following triviality criterion for the ideals $\widetilde{I}_p(D)$:

\begin{lemma}\label{tilde_triviality}
For every $p \ge 0$ we have
$$\widetilde{I}_p(D)=\shO_X \iff \partial_t^p\delta\in V^{\alpha}\iota_+\shO_X.$$
\end{lemma}
\begin{proof}
It is clear by definition that if $ \partial_t^p\delta\in V^{\alpha}\iota_+\shO_X$, then $1\in \widetilde{I}_p(D)$,
giving one implication. On the other hand, the converse is clear for $p=0$, and in general we argue by induction. 
If $\widetilde{I}_p(D)=\shO_X$, then  there is an element
$$v=\partial_t^p\delta+\sum_{j=0}^{p-1}v_j \partial_t^j\delta\in V^{\alpha}\iota_+\shO_X.$$
By considering $(f-t)^i v$, for $1\leq i\leq p$, we see that $\widetilde{I}_{p-i}(D)=\shO_X$, hence $\partial_t^{p-i}\delta\in V^{\alpha}\iota_+\shO_X$ by induction.
Therefore we have
$$\partial_t^p\delta=v-\sum_{j=1}^p v_{p-j}\partial_t^{p-j}\delta\in V^{\alpha}\iota_+\shO_X.$$
\end{proof}

Recall now from \cite{MP1} and \cite{MP3} the following notion which extends that of a log canonical pair.

\begin{definition}\label{klc}
The pair $(X, D)$ is \emph{$k$-log canonical} if 
$$I_0 (D) = \cdots = I_k (D) = \shO_X.$$
Corollary \ref{inclusions} implies that this is equivalent to $I_k (D) = \shO_X$.
Note that for this to hold, we need $\alpha\leq 1$. We make the convention that $(X,D)$ is $(-1)$-log canonical
if and only if $\alpha \le 1$.
\end{definition}

For the first nontrivial ideal we have a statement that is stronger than that of Corollary \ref{inclusions}.

\begin{corollary}\label{cor_first_consequence}
If $(X, D)$ is $(p-1)$-log canonical, then 
$$ I_p (D) = \widetilde{I}_p (D)$$ 
and also
$$I_{p+1} (D) \subseteq I_p (D).$$
In particular, we always have $I_1 (D) \subseteq I_0 (D)$ when $D = \alpha Z$ with $\alpha\leq 1$.
\end{corollary}

\begin{proof}
The inclusion $\widetilde{I}_p (D) \subseteq I_p (D)$  follows from the identity
$$I_p(D)+(f)=\widetilde{I}_p(D)+(f),$$ 
combined with the fact that $(f) \subseteq I_p (D)$ due to $(p-1)$-log canonicity; see assertion ii) in Remark \ref{alternative} 
(note that the inclusion also holds if $p=0$, by Remark~\ref{inclusion_in_I0}). To prove the opposite inclusion, it suffices to 
show that we also have $(f) \subseteq \widetilde{I}_p (D)$. To this end, the triviality of $I_{p-1}(D)$ implies that we also have 
$\widetilde{I}_{p-1}(D)=\shO_X$, which in turn is equivalent to 
$$\partial_t^{p-1}\delta\in V^{\alpha}\iota_+\shO_X$$
by Lemma \ref{tilde_triviality}. On the other hand, we have
$$f\partial_t^{p}\delta = t \partial_t^{p}\delta + p  \partial_t^{p-1}\delta
= t\partial_t\cdot  \partial_t^{p-1}\delta + p \partial_t^{p-1}\delta,$$
and so it follows that 
$$f\partial_t^{p}\delta\in V^{\alpha}\iota_+\shO_X$$
as well, which gives $f \in \widetilde{I}_p (D)$. This proves the first statement.

The second statement follows since by Corollary \ref{inclusions} we have
$$I_{p+1} (D) \subseteq I_{p+1} (D) + (f) \subseteq I_p (D) + (f) = I_p (D),$$
the last equality again being due to $(p-1)$-log canonicity.
\end{proof}

We also obtain information about the behavior of the Hodge ideals $I_p (\alpha Z)$ when $\alpha$ varies. In the case of $I_0$, via 
the connection with multiplier ideals (or directly from the description in terms of $V^\alpha \iota_+ \shO_X$), it is well known that they get smaller as $\alpha$ increases, and that there is a discrete set of values of $\alpha$ (called jumping coefficients) where the ideal actually changes; see \cite[Lemma ~9.3.21]{Lazarsfeld}. This is not the case for higher $k$; 
for instance, for the cusp 
$Z = (x^2 + y^3 = 0)$ and $\alpha \le 1$ and close to $1$, we see in \cite[Example 10.5]{MP3} that 
$$I_2 (\alpha Z) = (x^3, x^2y^2, xy^3, y^4- (2 \alpha + 1) x^2 y),$$
and thus we obtain incomparable ideals. However, Theorem \ref{general_description_variant}$^\prime$ 
implies that the picture does becomes similar to that for multiplier ideals if one considers the images in $\shO_Z$.

\begin{corollary}[Jumping coefficients]\label{cor_jumping_coeff}
Given any $p\geq 0$, there exists a finite set of rational numbers $0 = c_0 < c_1 < \cdots < c_s < c_{s+1} = 1$ such that 
for each $0 \le i \le s$ and each $\alpha \in (c_i, c_{i+1}]$ we have 
$$I_k (\alpha Z)\cdot\shO_Z = I_k (c_{i+1} Z)\cdot\shO_Z = {\rm constant}$$
and such that 
$$I_k (c_{i+1} Z)\cdot\shO_Z\subsetneq I_k (c_i Z)\cdot\shO_Z.$$
\end{corollary}

In fact, if $Z$ is defined by a global equation $f$,
the set of $c_i$ is a subset of the set of jumping numbers for the $V$-filtration on $\iota_+\shO_X$  
associated to $f$.

\begin{remark}
For $p=0$, we have 
$$\shO_X(-Z)={\mathcal I}(Z)\subseteq {\mathcal I}\big((\alpha-\epsilon)Z\big)=I_0(\alpha Z)$$
for every $\alpha\in (0,1]$, where $0<\epsilon\ll 1$. It follows that for $p=0$, the jumping coefficients in Corollary~\ref{cor_jumping_coeff}
coincide with those jumping coefficients for the multiplier ideals of $Z$,  in the sense of \cite{ELSV}, that lie in $(0,1]$.
\end{remark}

\begin{remark}\label{further_consequences}
Note that Theorem \ref{general_description_variant}$^\prime$ implies further facts about elements in the $V$-filtration on $\iota_+ \shO_X$.  
For example, if $v=\sum_{j=0}^pv_j\partial_t^j\delta \in V^{\alpha}\iota_+\shO_X$, with $p\geq 2$ and $\alpha<1$, then
\begin{equation}\label{other}
\sum_{j=2}^p(j-1)Q_{j-1}(\alpha)f^{p-j+1}v_j\in I_p(D).
\end{equation}
For $p=2$, this says that $fv_2\in I_2(D)$, so that $f \cdot \widetilde{I}_2 (D) \subseteq I_2 (D)$.
 
Indeed, it follows from Theorem~\ref{general_description_variant}$^\prime$ that
$$h:=\sum_{j=0}^pQ_j(\alpha)f^{p-j}v_j\in I_p(D).$$
Since 
$$(f-t)v=\sum_{j=1}^pjv_j\partial_t^{j-1}\delta\in V^{\alpha}\iota_+\shO_X,$$
another application of Theorem~\ref{general_description_variant}$^\prime$ gives
$$g:=\sum_{j=0}^{p-1}Q_j(\alpha)(j+1)f^{p-1-j}v_{j+1}\in I_{p-1}(D).$$
Therefore we have $fg\in I_p(D)$
(see assertion ii) in Remark~\ref{alternative}).
Note also that we always have $(f^{p+1})\subseteq I_p(D)$, 
by combining Remark~\ref{inclusion_in_I0} with the assertion ii) in Remark~\ref{alternative}.
We thus obtain 
$$(fh-f^{p+1}v_0)-\alpha fg\in I_p(D).$$
We now compute
\begin{align*}
(fh-f^{p+1}v_0)-\alpha fg&=\sum_{j=1}^pf^{p-j+1}v_j\big(Q_j(\alpha)-\alpha j Q_{j-1}(\alpha)\big)\\
&=(1-\alpha)\cdot\sum_{j=2}^p (j-1)Q_{j-1}(\alpha)f^{p-j+1}v_j.
\end{align*}
Dividing by $(1-\alpha)$, which is assumed to be nonzero, we obtain ($\ref{other}$).
\end{remark}

\subsection{Bernstein-Sato polynomials and minimal exponent}\label{refined_lct}
In this section we relate the $p$-log canonicity of a pair $(X,D )$, with $D = \alpha Z$, 
to the Bernstein-Sato polynomial of $Z$. We begin by recalling the definition and some basic facts about
Bernstein-Sato polynomials. 

Suppose that $X$ is a smooth complex variety and $f\in\shO_X(X)$ is a nonzero regular function on $X$. 
The \emph{Bernstein-Sato polynomial} $b_f(s)\in\CC[s]$ of $f$ is the (nonzero) monic polynomial of minimal degree such that 
\begin{equation}\label{eq_b_function}
b_f(s)f^s\in \Dmod_X[s]\cdot f^{s+1}.
\end{equation}
If $f$ is not invertible, by setting $s=-1$ in (\ref{eq_b_function}),
we see that $(s+1)$ divides $b_f(s)$. We can thus write $b_f(s)=(s+1)\cdot \widetilde{b}_f(s)$, and $\widetilde{b}_f(s)$
is called the \emph{reduced Bernstein-Sato polynomial} of $f$. This invariant was studied by Saito in \cite{Saito_microlocal}. In particular,
he showed that it is related to the microlocal $V$-filtration mentioned in Remark~\ref{microlocal_V_filtration}; consequently, $\widetilde{b}_f(s)$ was also called the \emph{microlocal $b$-function} in \emph{loc.cit.} 

The existence of a nonzero polynomial $b_f(s)$ that satisfies (\ref{eq_b_function}) was proved by Bernstein \cite{Bernstein} when $X={\mathbf A}^n$. 
For a proof in the case of arbitrary $X$ (or, more generally, when $f$ is a holomorphic function on a complex manifold), see 
\cite{Kashiwara2} and \cite{Bjork}.
It follows from the definition that if $X=\bigcup_{i\in I}U_i$ is a finite open cover, then $b_f(s)$ is the least common multiple of the polynomials 
$(b_{f\vert_{U_i}})_{i\in I}$. Moreover, one can show that if $g$ is an invertible function, then $b_f(s)=b_{fg}(s)$. If $E$ is an effective divisor on $X$, we can thus define the Bernstein-Sato polynomial $b_E(s)$ such that if $X=\bigcup_{i\in I}U_i$ is a finite open cover and $f_i\in\shO_X(U_i)$ is an
equation of $E\vert_{U_i}$, then $b_E(s)$ is 
the least common multiple of the polynomials $\big(b_{f_i}(s)\big)_{i\in I}$. If $E\neq 0$, then $b_E(s)=(s+1)\cdot \widetilde{b}_E(s)$, for a polynomial $\widetilde{b}_E(s)$. 

It is sometimes convenient to consider a local version. It is easy to see that for every $x\in X$ and every effective divisor $E$ on $X$,
there is an open neighborhood $U$ of $x$ such that $b_{E\vert_U}(s)$ divides $b_{E\vert_V}(s)$ for every other such neighborhood $V$. We set 
$$b_{E,x}(s):=b_{E\vert_U}(s).$$ 
Note that if $x\in E$, then $(s+1)$ divides $b_{E,x}(s)$; the quotient is denoted $\widetilde{b}_{E,x}(s)$.

By a result of Kashiwara \cite{Kashiwara2}, for every effective divisor $E$ on $X$, 
all roots of $b_E(s)$ are negative rational numbers. The negative of the largest
root of $b_E(s)$ is an important invariant of singularities, the \emph{log canonical threshold} $\alpha_E$, also denoted ${\rm lct}(X,E)$ (see
\cite[Theorem~10.6]{Kollar}). Assuming $E\neq 0$, we can also consider a refined version of the log canonical threshold, 
denoted $\widetilde{\alpha}_E$, which is the negative of the largest root of $\widetilde{b}_E(s)$; we call this the \emph{minimal exponent} 
of $E$, following \cite{Saito-B} (it is also called the \emph{microlocal log canonical threshold} in \cite{Saito-MLCT}). 
We make the convention that if $\widetilde{b}_E (s)$ is a constant, then $\widetilde{\alpha}_E=\infty$. Note that we have 
$$\alpha_E=\min\{1,\widetilde{\alpha}_E\}.$$
If $E$ is defined by $f\in\shO_X(X)$, then we also write $\widetilde{\alpha}_f$ for $\widetilde{\alpha}_E$. 
We can similarly define local versions of these invariants: given $x\in E$, the log canonical threshold $\alpha_{E,x}$
is the negative of the largest root of $b_{E,x}(s)$ and $\widetilde{\alpha}_{E,x}$
is the negative of the largest root of $\widetilde{b}_{E,x}(s)$. 
When $E$ has an isolated singularity at $x$, the invariant $\widetilde{\alpha}_{E,x}$ is also known as the 
\emph{complex singularity index} of $E$ at $x$.

Our main result implies that the minimal exponent governs the $p$-log canonicity of $(X, \alpha Z)$. Since we have observed in Definition \ref{klc} that this $p$-log canonicity condition is equivalent to  $I_p (\alpha Z) = \shO_X$, the first statement below is equivalent to Corollary \ref{mlc_intro} in the introduction.

\begin{corollary}\label{mlc}
If $Z\neq 0$ is a reduced effective divisor on the smooth variety $X$  and $\alpha\in (0,1]$ is a rational number, then the pair $(X, \alpha Z)$ is $p$-log canonical if and only if 
$$p\leq\widetilde{\alpha}_Z - \alpha.$$
Similarly, the pair $(X,\alpha Z)$ is $p$-log canonical in some neighborhood of $x\in Z$ if and only if $p\leq\widetilde{\alpha}_{Z,x} - \alpha$.
\end{corollary}
\begin{proof}
For $\alpha=1$, this is due to Saito \cite{Saito-MLCT}. The proof combines the connection between 
Hodge ideals and the microlocal $V$-filtration in \emph{loc. cit.} with a result  deduced from \cite{Saito_microlocal} 
relating $\widetilde{\alpha}_f$ to the latter, where $f$ is a local equation defining $Z$; namely
$$\widetilde{\alpha}_f = {\rm max}~\{\gamma \in \QQ~|~ \widetilde{V}^\gamma \shO_X = \shO_X\},$$
see \cite[(1.3.8)]{Saito-MLCT}. (Note that by Nakayama's Lemma the triviality of $\widetilde{V}^\gamma$ 
at the points of $Z$ is equivalent to the triviality of $\widetilde{V}^\gamma\cdot \shO_Z$.)

Once we have Theorem \ref{general_description_variant}$^\prime$, the exact same argument applies in the setting of the above corollary; see also Remark \ref{microlocal_V_filtration}.
\end{proof}

\medskip

Combining Corollary \ref{mlc} with results derived from the birational study of Hodge ideals in \cite{MP3},  
we obtain the estimate for $\widetilde{\alpha}_Z$ in terms of a log resolution of $(X,Z)$ in Corollary \ref{lichtin_inequality}. We fix such a log resolution, i.e. a proper birational morphism $\mu \colon Y\to X$, 
with $Y$ smooth, such that $\mu^*Z$ has simple normal crossings support. 
We assume in addition that $\mu$ is an isomorphism over $X\smallsetminus Z$ and that the strict transform $\widetilde{Z}$ of $Z$ is smooth.
Let $F_1,\ldots,F_m$ be the irreducible components of the exceptional locus of $\mu$ and write
$$\mu^* Z =\widetilde{Z}+\sum_{i=1}^ma_iF_i\quad\text{and}\quad K_{Y/X}=\sum_{i=1}^mb_iF_i.$$
Recall that we denote
$$\gamma : = \underset{i=1, \ldots, m}{\rm min} \left\{\frac{b_i +1}{a_i}\right\}.$$
The log canonical threshold of $(X, Z)$ is given by ${\alpha}_Z=\min\{\gamma,1\}$, and we also 
have ${\alpha}_Z=\min\{\widetilde{\alpha}_Z,1\}$. We now show the inequality $\widetilde{\alpha}_Z\geq \gamma$; see  the Introduction for a discussion.

\begin{proof}[Proof of Corollary~\ref{lichtin_inequality}]
Given any rational number $\alpha\in (0,1]$, it follows from \cite[Proposition 11.2]{MP3} that 
if $\gamma\geq p+\alpha$, then $I_p(\alpha Z)=\shO_X$. We deduce from 
Corollary~\ref{mlc} that  we also have $\widetilde{\alpha}_Z\geq p+\alpha$. 

By taking $p=\lceil\gamma\rceil -1$ and $\alpha=\gamma+1-\lceil\gamma\rceil$, we 
have $\alpha\in (0,1]$ and $p+\alpha=\gamma$, hence we obtain $\widetilde{\alpha}_Z\geq\gamma$.
\end{proof}

\begin{remark}[Rational singularities]\label{rmk_rational_sing}
Saito showed in \cite[Theorem~0.4]{Saito-B} that an integral effective divisor $D$ on $X$ has rational singularities
if and only if $\widetilde{\alpha}_D>1$. The ``only if" part also follows from Corollary~\ref{lichtin_inequality}, since it is known that $D$ has rational singularities
if and only if $\gamma>1$ (see \cite[Theorems~7.9 and 11.1]{Kollar}). In order to handle the ``if" part via Corollary~\ref{mlc}, one needs to show that if
$I_1(\alpha D)=\shO_X$ for some $\alpha\in (0,1]$, then $D$ has rational singularities. (Note that if $\widetilde{\alpha}_D>1$, then $D$ is automatically reduced:
otherwise the log canonical threshold is $\leq 1/2$, and thus $\widetilde{\alpha}_D=\alpha_D\leq 1/2$.) 
Since a reduced divisor $D$ has rational singularities if and only if ${\rm adj}(D)=\shO_X$, where ${\rm adj}(D)$ is the adjoint ideal of $D$ (see \cite[Proposition~9.3.48]{Lazarsfeld}),
we see that the ``if" part of the above assertion would follow from a positive answer to the following question.
\end{remark}

\begin{question}
If $Z$ is a reduced effective  divisor on the smooth variety $X$ and $\alpha$ is a rational number in $(0,1]$, do we have the inclusion
$$I_1(\alpha Z)\subseteq {\rm adj}(Z)?$$
For $\alpha=1$, a positive answer is provided by \cite[Theorem~C]{MP1}.
\end{question}

We now turn to the general properties of the minimal exponent stated in the introduction. We use  basic facts about Hodge ideals established in \cite{MP3}.

\begin{proof}[Proof of Theorem~\ref{properties_refined_lct}]
For the assertion in (1), we may assume that $Y$ is a divisor in $X$. Indeed, if $r={\rm codim}_X(Y)$, then 
after possibly replacing $X$ by an open neighborhood of $x$, we can find smooth, irreducible subvarieties $Y_0=X,Y_1,\ldots,Y_r=Y$ of $X$
such that $Y_i$ is a divisor in $Y_{i-1}$ for $1\leq i\leq r$. If we know the assertion for $r=1$, we obtain
$$\widetilde{\alpha}_{D\vert_Y,x}\leq \widetilde{\alpha}_{D\vert_{Y_{r-1},x}}\leq\cdots\leq\widetilde{\alpha}_{D,x}.$$
From now on, we assume that $Y$ is a divisor in $X$.

We may also assume that
$D\vert_Y$ is reduced in a neighborhood of $x$. 
Indeed, otherwise we have $\alpha_{D\vert_Y,x}\leq \frac{1}{2}$, hence $\widetilde{\alpha}_{D\vert_Y,x}=\alpha_{D\vert_Y,x}$, and we use the fact that for log canonical thresholds the analogue of (1) is known. For example, this follows using the interpretation of the log canonical threshold in terms of multiplier ideals, combined with the Restriction Theorem for such ideals, see
\cite[Theorem~9.5.1]{Lazarsfeld};
we thus have
$$\widetilde{\alpha}_{D\vert_Y,x}=\alpha_{D\vert_Y,x}\leq\alpha_{D,x}\leq\widetilde{\alpha}_{D,x}.$$
After replacing $X$ by a suitable neighborhood of $x$, we may therefore assume that both $D$ and 
$D\vert_Y$ 
are reduced divisors. In this case the Restriction Theorem for Hodge ideals \cite[Theorem~13.1]{MP3} gives
$$I_p(\alpha D\vert_Y)\subseteq I_p(\alpha D)\cdot\shO_Y$$
for every non-negative integer $p$ and every positive rational number $\alpha$. 
By taking $p=\lceil\widetilde{\alpha}_{D\vert_Y,x}\rceil-1$ and $\alpha=\widetilde{\alpha}_{D\vert_Y,x}-p\in (0,1]$,
it follows from Corollary~\ref{mlc} that $I_p(\alpha D\vert_Y)_x=\shO_{Y,x}$, hence by the inclusion above we also have $I_p(\alpha D)_x=\shO_{X,x}$.
Another application of Corollary~\ref{mlc} then gives $\widetilde{\alpha}_{D\vert_Y,x}\leq \widetilde{\alpha}_{D,x}$.

In order to prove the semicontinuity statement in (2), we need to show that for every $t$ in $T$ there is an open neighborhood $U$ of $t$ such that
\begin{equation}\label{eq_semicont}
\widetilde{\alpha}_{D_{t'},s(t')}\geq\widetilde{\alpha}_{D_t,s(t)}\quad\text{for every}\quad t'\in U.
\end{equation}
If $D_t$ is not reduced, then arguing as above we see that $\widetilde{\alpha}_{D_t,s(t)}=\alpha_{D_t,s(t)}$.
The semicontinuity property of log canonical thresholds (see  \cite[Example~9.5.41]{Lazarsfeld}) implies then that there is an open neighborhood $U$ of $t$ such that
$$\widetilde{\alpha}_{D_{t'},s(t')}\geq \alpha_{D_{t'},s(t')}\geq \alpha_{D_t,s(t)}\quad \text{for every}\quad t'\in U,$$
which gives (\ref{eq_semicont}). Suppose now that $D_t$ is reduced. After possibly replacing $T$ by an open neighborhood $T'$ of $t$, and $X$ by $\pi^{-1}(T')$,
we may assume that $D_{t'}$ is reduced for all $t'\in T$; in particular, $D$ is reduced as well.  In this case, the Semicontinuity Theorem for Hodge ideals 
\cite[Theorem~14.1]{MP3} applies; it says that for every $p\geq 0$ and every positive rational number $\alpha$, if 
$I_p(\alpha D)_{s(t)}=\shO_{X_t,s(t)}$, then there is an open neighborhood $U$ of $t$ such that 
$I_p(\alpha D)_{s(t')}=\shO_{X_{t'},s(t')}$ for every $t'\in U$. Taking $p=\lceil\widetilde{\alpha}_{D_t,s(t)}\rceil-1$ and 
$\alpha=\widetilde{\alpha}_{D_t,s(t)}-p\in (0,1]$,  it follows from Corollary~\ref{mlc} that $I_p(\alpha D)_{s(t)}=\shO_{X_t,s(t)}$.
Another application of the corollary gives (\ref{eq_semicont}) on $U$.

In order to prove (3), we may assume that $D$ is reduced in a neighborhood of $x$. Indeed, otherwise as before we have 
 $\widetilde{\alpha}_{D,x}=\alpha_{D,x}$ and also $r = n-2$. However, for the log canonical threshold the bounds
 $$\frac{1}{m} \le \alpha_{D,x} \le \frac{n}{m}$$
 are well known and easy to prove (see e.g. \cite[Lemma~8.10]{Kollar}). 
 After passing to such a neighborhood, we may thus assume that $D$ is reduced. 

In this case, it follows from \cite[Corollary~11.11]{MP3} that $I_p(\alpha D)_x\neq\shO_{X,x}$ if $(\alpha+p) m>n$. If $\alpha\leq 1$, then we conclude from
Corollary~\ref{mlc} that $p>\widetilde{\alpha}_{D,x}-\alpha$. If $\widetilde{\alpha}_{D,x}>\frac{n}{m}$, then by taking $p=\lceil \widetilde{\alpha}_{D,x}\rceil-1\geq 0$ and $\alpha=
\widetilde{\alpha}_{D,x}-p\in (0,1]$, we obtain a contradiction. This proves the upper bound. 

To prove the lower bound,  we may also assume that  $X$ is affine, and we have an algebraic system of coordinates 
$x_1,\ldots,x_n$ on $X$, centered at $x$. If $H$ is defined by a general linear combination of $x_1,\ldots,x_n$, then $H$ is smooth and irreducible in a suitable neighborhood of $x$.
Furthermore, $H$ is not contained in $D$, we have ${\rm mult}_x(D\vert_H)=m$, and ${\mathbf P}\big(C_x(D\vert_H)\big)$ is a general hyperplane section of 
${\mathbf P}(C_xD)$; in particular, the singular locus of ${\mathbf P}\big(C_x(D\vert_H)\big)$ has dimension $r-1$. Since 
$\widetilde{\alpha}_{D\vert_H,x}\leq\widetilde{\alpha}_{D,x}$ by part (1), we see that it is enough to prove the lower bound for 
$\widetilde{\alpha}_{D\vert_H,x}$. After $r+1$ such steps, we reduce to the case when $r=-1$, that is, ${\mathbf P}(C_xD)$ is smooth. In this case, if we take $p=\lceil \frac{n}{m}\rceil-1$ and $\alpha=\frac{n}{m}-p\in (0,1]$, then it follows from \cite[Example~11.6]{MP3} that
$I_p(\alpha D)_x=\shO_{X,x}$, and we conclude using Corollary~\ref{mlc} that
$$\frac{n}{m}=p+\alpha\leq\widetilde{\alpha}_{D,x}.$$
This completes the proof of (3).
\end{proof}

\begin{remark}
It is straightforward to see that if $x$ is a smooth point of $D$, then $b_{D,x}(s)=s+1$, hence $\widetilde{\alpha}_{D,x}=\infty$. On the other hand,
if $x$ is a singular point of $D$, then it follows from part (3) in Theorem~\ref{properties_refined_lct} that 
$\widetilde{\alpha}_{D,x}\leq \frac{n}{2}$.
This also follows from \cite[Theorem~0.4]{Saito_microlocal}, which asserts moreover that the negative of every root of $\widetilde{b}_{D,x}(s)$ lies in the closed interval $[\widetilde{\alpha}_{D,x},n-\widetilde{\alpha}_{D,x}]$. 
\end{remark}

\begin{remark}
Part (2) in Theorem \ref{properties_refined_lct} can also be deduced from (1) using the invariance of the minimal exponent under non-characteristic restriction, which follows from results in \cite{DMST}; see \cite[Remark~1.3 (iv)]{JKSY}. 
\end{remark}

In the next proposition we collect  further properties of the minimal exponent that can be deduced with the help of  Theorem~\ref{properties_refined_lct}. For the corresponding results for log canonical thresholds, see \cite[\S8]{Kollar}.

\begin{proposition}\label{further_properties}
Let $X$ be a smooth $n$-dimensional variety.
\begin{enumerate}
\item  If $f,g\in\shO_X(X)$ are such that $f$, $g$, and $f+g$ are nonzero, then for every $x\in X$ such that $f(x)= g(x) =0$ we have
$$\widetilde{\alpha}_{f+g,x}\leq \widetilde{\alpha}_{f,x}+\widetilde{\alpha}_{g,x}.$$
\item If $f,g\in\shO_X(X)$ are nonzero and $x\in X$ is such that $f(x)= g(x) =0$ and ${\rm mult}_x(f-g)=d\geq 2$, then
$$|\widetilde{\alpha}_{f,x}-\widetilde{\alpha}_{g,x}|\leq\frac{n}{d}.$$
\item If $f\in\shO_X(X)$ is nonzero and $x\in X$ is such that $f(x)=0$, then for every sequence
$(f_i)_{i\geq 1}$ with $f_i\in\shO_X(X)$, such that $\lim_{i\to\infty}{\rm mult}_x(f_i-f)=\infty$, we have
$$\widetilde{\alpha}_{f,x}=\lim_{i\to\infty}\widetilde{\alpha}_{f_i,x}.$$
\end{enumerate}
\end{proposition}

The key input for the proof of the proposition is the following special case, due to Saito.

\begin{example}\label{TS_theorem}
Let $X$ and $Y$ be smooth varieties and $f\in\shO_X(X)$, $g\in\shO_Y(Y)$ be nonzero regular functions. 
Consider the two projections $\pi_1\colon X\times Y\to X$ and $\pi_2\colon X\times Y\to Y$. 
If $x\in X$ and $y\in Y$ are such that $f(x)=0$ and $g(y)=0$, then
\begin{equation}\label{eq0_TS_theorem}
\widetilde{\alpha}_{f\oplus g,(x,y)}=\widetilde{\alpha}_{f,x}+\widetilde{\alpha}_{g,y},
\end{equation}
where $f\oplus g=f\circ \pi_1+g\circ\pi_2$. 
This is a consequence of the Thom-Sebastiani property for \emph{microlocal multiplier ideals} proved in \cite[Theorem~2.2]{MSS}
and of Saito's description of the minimal exponent via the microlocal $V$-filtration as in the proof of Corollary \ref{mlc} 
(cf. also Corollary~\ref{mlc_intro} and Remark~\ref{microlocal_V_filtration}), namely: 
$$\widetilde{\alpha}_{f,x}=\max\{\gamma>0\mid 1\in \widetilde{V}^{\gamma}\shO_X\,\,\text{in a neighborhood of}\,x\}.$$
\end{example}

\begin{proof}[Proof of Proposition~\ref{further_properties}]
The assertion in (1) follows by applying the inequality in Theorem~\ref{properties_refined_lct} (1) to the diagonal embedding 
$X\hookrightarrow X\times X$ and to $f\oplus g$, and using the formula for $\widetilde{\alpha}_{f\oplus g,(x,y)}$ in Example~\ref{TS_theorem}. We deduce the inequality in (2) using (1) and the fact that (assuming $f\neq g$), we have $\widetilde{\alpha}_{f-g,x}\leq\frac{n}{d}$ by Theorem~\ref{properties_refined_lct} (3). Finally, (3) is an immediate consequence of (2).
\end{proof}

\begin{example}
Recall that if $X$ is smooth and $f\in\shO_X(X)$ is nonzero, a result of Saito says that the hypersurface defined by $f$ is rational in the neighborhood of some $x\in X$ with $f(x)=0$, if and only if $\widetilde{\alpha}_{f,x}>1$ (see Remark~\ref{rmk_rational_sing}). An amusing consequence of 
Proposition~\ref{further_properties} (3) is that if this is the case, then for every sequence
$(f_i)_{i\geq 1}$ with $f_i\in\shO_X(X)$, such that $\lim_{i\to\infty}{\rm mult}_x(f_i-f)=\infty$,
the hypersurface defined by $f_i$ has rational singularities in a neighborhood of $x$, for $i\gg 0$. 
\end{example}

In the spirit of the analogy with the behavior of log canonical thresholds, we ask further questions regarding the behavior of minimal exponents.

\begin{question}\label{q_ACC}
Let $n\geq 1$ be fixed and consider the set $\widetilde{\mathcal T}_n$ consisting of all rational numbers $\widetilde{\alpha}_D$, 
where $D$ is a nonzero effective divisor on a smooth $n$-dimensional variety. Does the set $\widetilde{\mathcal{T}}_n$ satisfy ACC, that is, does it contain no infinite strictly increasing sequences?

Note that the set ${\mathcal T}_n=\widetilde{\mathcal T}_n\cap (0,1]$ consists precisely of the set of log canonical thresholds for divisors on smooth $n$-dimensional varieties. This set is known to satisfy ACC: this was a conjecture of Shokurov, proved in \cite{dFEM}.
\end{question}

\begin{question}
Suppose that $X$ is a smooth variety, $f\in\shO_X(X)$ is nonzero, and $x\in X$ such that $f(x)=0$.
Is it true that for every sequence $(f_i)_{i\geq 1}$ with $f_i\in\shO_X(X)$, such that $\lim_{i\to\infty}{\rm mult}_x(f_i-f)=\infty$,
we have
$$\widetilde{\alpha}_{f_i,x}\geq\widetilde{\alpha}_{f,x}\quad\text{for all}\quad i\gg 0?$$
Note that by Proposition~\ref{further_properties} (3), a positive answer to Question~\ref{q_ACC} implies a positive answer to this question as well. It is worth noting, however, that when dealing with log canonical thresholds, the proof of the ACC property in \cite{dFEM} proceeds by first proving the analogue of this weaker question.
\end{question}

\medskip

We conclude by showing that the negatives of the jumping coefficients introduced in Corollary~\ref{cor_jumping_coeff} give, under a suitable condition, roots of the Bernstein-Sato polynomial.
We accomplish this with the help of a result of general interest regarding Bernstein-Sato polynomials of certain elements in $ \iota_+\shO_X$, Proposition \ref{prop_Bernstein_poly} below, which we hope will be useful
in other contexts as well. We also make use of Sabbah's description of the $V$-filtration in terms of such polynomials.

We start by recalling these concepts, using the notation in \S\ref{scn:Vfil}.
Given an element $u\in  \iota_+\shO_X$, the \emph{Bernstein-Sato polynomial} $b_u(s)$ is the (nonzero)
monic polynomial of smallest degree such that
$$b_u(-\partial_tt)u\in \Dmod_X\langle\partial_tt, t\rangle\cdot tu.$$
Using Proposition~\ref{alternative_description} and the fact that $t^j\delta=f^j\delta$ for all $j\geq 1$, it follows
that $b_{\delta}(s)$ is the same as $b_f(s)$. The following result, due to Sabbah, gives a description
of the $V$-filtration on $\iota_+\shO_X$ in terms of Bernstein-Sato polynomials. 
We note that in the case $\Mmod=\shO_X$, the existence of $b_u(s)$ and the rationality of its roots
follows easily from the existence of the
$V$-filtration on $\iota_+\Mmod$, which in turn was constructed in \cite{Malgrange} starting from the existence of $b_f(s)$.\footnote{For more general $\Dmod_X$-modules $\Mmod$, one first proves the existence of general Bernstein-Sato polynomials and then uses this to construct the $V$-filtration on $\iota_+\Mmod$.}

\begin{proposition}[\cite{Sabbah}]\label{sabbah}
For every $\gamma\in\QQ$, we have
$$V^{\gamma}\iota_+\shO_X = \{ u \in  \iota_+\shO_X ~|~  b_u(s) {\rm ~has ~all~ roots~}  \leq -\gamma \}.$$
\end{proposition}

We use this proposition, as well as the relationship between  $\widetilde{b}_f$ and the microlocal $V$-filtration, to deduce the following relation between $\widetilde{b}_f(s)$ and the polynomials $b_{\partial_t^m\delta}(s)$. 

\begin{proposition}\label{prop_Bernstein_poly}
For every nonnegative integer $m$, we have the following divisibility properties of polynomials in $\CC[s]$:
$$b_{\partial_t^m\delta}(s)\vert (s+1)\widetilde{b}_f(s-m)\quad\text{and}\quad \widetilde{b}_f(s-m)\vert b_{\partial_t^m\delta}(s).$$
\end{proposition}

\begin{proof}
We may and will assume that $X$ is affine.
We begin by noting that for every polynomial $Q(s)$, we have
\begin{equation}\label{eq1_prop_Bernstein_poly}
\partial_t\cdot Q(\partial_tt)=Q(\partial_tt+1)\cdot\partial_t\quad\text{and}\quad t\cdot Q(\partial_tt)=Q(\partial_tt-1)\cdot t.
\end{equation}
Indeed, it is enough to check this when $Q(s)=s^q$ is a monomial, and in this case both equalities can be easily verified by induction on $q$.

By the definition of the Bernstein-Sato polynomial $b_f(s)=b_{\delta}(s)$, we can find $P\in\Dmod_X(X)[s]$ such that
$$b_f(-\partial_tt)\delta=P(-\partial_tt)t\delta.$$
Using (\ref{eq1_prop_Bernstein_poly}), we obtain
$$P(-\partial_tt)t\delta=tP(-\partial_tt-1)\delta\quad\text{and}$$
$$b_f(-\partial_tt)=(1-\partial_tt)\widetilde{b}_f(-\partial_tt)=-\widetilde{b}_f(-\partial_tt)t\partial_t=-t\cdot\widetilde{b}_f(-\partial_tt-1)\partial_t.$$
Since the action of $t$ on $\iota_+\shO_X$ is injective, we deduce
\begin{equation}\label{eq2_prop_Bernstein_poly}
\widetilde{b}_f(-\partial_tt-1)\partial_t\delta=R(-\partial_tt)\delta,\quad\text{where}\quad R(s)=-P(s-1).
\end{equation}
Using (\ref{eq1_prop_Bernstein_poly}), we get
$$\widetilde{b}_f(-\partial_tt-m)\partial_t^m\delta=\partial_t^{m-1}\cdot\widetilde{b}_f(-\partial_tt-1)\partial_t\delta=
\partial_t^{m-1}\cdot R(-\partial_tt)\delta$$
$$=R(-\partial_tt-m+1)\partial_t^{m-1}\delta,$$
hence
$$(1-\partial_tt)\widetilde{b}_f(-\partial_tt-m)\partial_t^m\delta=R(-\partial_tt-m+1)\cdot (1-\partial_tt)\partial_t^{m-1}\delta$$
$$=
-R(-\partial_tt-m+1)\cdot t\partial_t^m\delta\in\Dmod_X[-\partial_tt]\cdot t\partial_t^m\delta.$$
By the definition of $b_{\partial_t^m\delta}(s)$, we thus conclude that 
$$b_{\partial_t^m\delta}(s)\vert (s+1)\widetilde{b}_f(s-m).$$

For the proof of the second divisibility relation, we make use of a result of Saito describing
$\widetilde{b}_f$ in terms of the microlocal $V$-filtration. For this, we consider the localization
$\widetilde{\mathcal R}:=\Dmod_X\langle t,\partial_t,\partial_t^{-1}\rangle$ of $\Dmod_X\langle t,\partial_t\rangle$ with respect to $\partial_t$.
Similarly, we consider the localization $\widetilde{B}_f$ of $\iota_+\shO_X$ with respect to $\partial_t$, so that
$$\widetilde{B}_f=\bigoplus_{j\in\ZZ}\shO_X\partial_t^j\delta.$$
(See \cite{Saito_microlocal} for more details about this construction.)
It was shown in \cite[Proposition~0.3]{Saito_microlocal} that $\widetilde{b}_f$ is the monic polynomial of smallest degree such that 
$\widetilde{b}_f(-\partial_tt)\delta\in V^1\widetilde{\mathcal R}\cdot\delta$, where
for every $p\in\ZZ$, we put
$$V^p\widetilde{\mathcal R}=\bigoplus_{i-j\geq p}\Dmod_Xt^i\partial_t^j.$$
Note that $V^p\widetilde{\mathcal R}=\partial_t^{-p}\cdot V^0\widetilde{\mathcal R}=V^0\widetilde{\mathcal R}\cdot \partial_t^{-p}$ for all $p\in \ZZ$. 

If $b(s)=b_{\partial_t^m\delta}(s)$, then by assumption there is $P\in\Dmod_X\langle \partial_tt,t\rangle$ such that
$$b(-\partial_tt)\partial_t^m\delta=P\cdot t\partial_t^m\delta.$$
Using (\ref{eq1_prop_Bernstein_poly}), we thus see that
$$\partial_t^mb(-\partial_tt+m)\delta\in \Dmod_X\langle \partial_tt,t\rangle \cdot t\partial_t^m\delta,$$
hence $b(-\partial_tt+m)\delta\in \partial_t^{-m}V^0\widetilde{\mathcal R}t\partial_t^{m}\cdot\delta\subseteq V^1\widetilde{\mathcal R}\cdot\delta$.
Saito's result mentioned above thus implies that $\widetilde{b}_f(s)$ divides $b(s+m)$.
\end{proof}

\begin{remark}
Note that the result above provides another approach to Corollary \ref{mlc}. Recall that by 
Theorem~\ref{general_description_variant}$^\prime$ the pair $(X,D)$ is $p$-log canonical if and only if
$\widetilde{I}_p(D)=\shO_X$. We may assume that $Z$ is defined by $f\in\shO_X(X)$. Now by Lemma \ref{tilde_triviality}, we have
$$\widetilde{I}_p(D)=\shO_X \iff  \partial_t^p\delta\in V^{\alpha}\iota_+\shO_X.$$
On the other hand, by Proposition \ref{sabbah} we see that
$\partial_t^p\delta\in V^{\alpha}\iota_+\shO_X$ if and only if all roots of $b_{\partial_t^p\delta}(s)$ are $\leq -\alpha$. Since
$\alpha\leq 1$, it follows from Proposition~\ref{prop_Bernstein_poly} that this condition holds if and only if all roots of $\widetilde{b}_f(s)$ are
$\leq -\alpha-p$, which is equivalent to $p\leq \widetilde{\alpha}_f-\alpha$. 
\end{remark}

We now come to our goal of relating jumping coefficients for Hodge ideals to roots of the Bernstein-Sato polynomial. This extends the assertion in \cite[Theorem~B]{ELSV}, which is the case $p=0$. 

\begin{proposition}\label{roots_b}
Let $Z\neq 0$ be a reduced, effective divisor on the smooth variety $X$ and suppose that
$\alpha\in (0,1)$ is a rational number and $p\geq 0$ is an integer such that
the pair $(X,\beta Z)$ is $(p-1)$-log canonical for some $\beta\in (\alpha, 1)$. If
$I_p(\alpha Z)\neq I_{p}\big((\alpha +\epsilon)Z\big)$ for $0<\epsilon\ll1$, then we have $\widetilde{b}_Z(-p-\alpha)=0$. 
\end{proposition}

\begin{proof}
We may assume that $X$ is affine and $Z$ is defined by $f\in\shO_X(X)$. In order to simplify the notation, we
 write $V^{\alpha}$ for $V^{\alpha}\iota_+\shO_X$.
Note first that since we assume that the pair $(X,\beta Z)$ is $(p-1)$-log canonical, 
we have $\shO_X(-Z)\subseteq I_p(\gamma Z)$ for every $\gamma\in (0,\beta]$ (see assertion ii) in Remark~\ref{alternative}), hence
our hypothesis on 
$\alpha$ is equivalent to  the condition that 
$$I_p(\alpha Z)\cdot\shO_Z\neq I_{p}\big((\alpha +\epsilon)Z\big)\cdot\shO_Z$$
for  $0<\epsilon\ll 1$. This is further equivalent to the existence of an $h\in\shO_X(X)$ such that
$h\partial_t^p\delta\in V^{\alpha}\smallsetminus V^{>\alpha}$; this follows using Corollary~\ref{cor_first_consequence}
and the fact that $\partial_t^j\delta\in V^{\beta}$ for $j\leq p-1$ by Lemma~\ref{tilde_triviality}. 

By the definition of general Bernstein-Sato polynomials, we have 
$$b_{\partial_t^p\delta}(-\partial_tt)\partial_t^p\delta\in \Dmod_X\langle \partial_tt, t\rangle \cdot t\partial_t^p\delta\subseteq V^{\beta},$$
where the inclusion follows from the fact that $\partial_t^{p-1}\delta\in V^{\beta}$. In particular, we have
$$b_{\partial_t^p\delta}(-\partial_tt)h\partial_t^p\delta\in V^{\beta}\subseteq V^{>\alpha}.$$
On the other hand, by the definition of the $V$-filtration, for $N\gg 0$ we have
$$(\partial_tt-\alpha)^Nh\partial_t^p\delta\in V^{>\alpha}.$$
If the two polynomials $b_{\partial_t^p\delta}(s)$ and $(s+\alpha)^N$ were coprime, we would infer that 
$h\partial_t^p\delta \in V^{>\alpha}$, which is a contradiction. Thus we deduce that
$b_{\partial_t^p\delta}(-\alpha)=0$. Since $\alpha\neq 1$, we conclude using Proposition~\ref{prop_Bernstein_poly}
that $\widetilde{b}_f(- p -\alpha)=0$. 
\end{proof}

\begin{remark}
M. Saito points out that Proposition \ref{roots_b} can also be obtained by combining the proof of Corollary \ref{cor_first_consequence} with the theory of microlocal Bernstein-Sato polynomials \cite{Saito_microlocal}, without appealing to the statement of Proposition \ref{prop_Bernstein_poly} (which does use this theory in its proof).
\end{remark}

\begin{example}
Let $Z \subset \CC^2$ be the cusp, defined by $f = x^2 + y^3$. It is well known that 
$$\tilde{b}_Z (s) = \left(s + \frac{5}{6}\right)\left(s+ \frac{7}{6}\right),$$
so that $\widetilde{\alpha}_Z = 5/6$ and $I_0 (\beta Z) = \shO_X$ for every $\beta \le 5/6$. On the other hand, explicit formulas for weighted homogeneous polynomials show that $I_1 \left(\frac{1}{6} Z\right) \neq I_1 \left((\frac{1}{6} + \epsilon) Z\right)$ for $0<\epsilon\ll 1$; see \cite[Example~3.5]{Zhang}. Thus the ``other" root $- 7/6 = -1 - 1/6$ is accounted for by the jumping number $1/6$ of $I_1$, as in Proposition \ref{roots_b}.
\end{example}

\subsection{Appendix: some combinatorial formulas}
In this appendix we derive some identities involving the polynomials 
$$Q_i(x)=i!\cdot {{x+i-1}\choose i}:=\prod_{j=0}^{i-1}(x+j)\in\ZZ[x], \,\,\,\,\,\,{\rm for}~ i\ge 0$$
(with the convention $Q_0=1$), used in the main body of the paper.

\begin{lemma}\label{eq2_precise_formula}
For every $j\geq 0$, we have
$$
Q_j(x)=\sum_{i=0}^j {j\choose i}Q_{j-i}(y)Q_i(x-y)\quad\text{in}\quad\ZZ[x,y].
$$
\end{lemma}
\begin{proof}
It is enough to show that 
$$
Q_j(k)=\sum_{i=0}^j {j\choose i}Q_{j-i}(\ell)Q_i(k-\ell)\quad\text{for every}\quad k,\ell\in\ZZ, 1\leq \ell< k.
$$
Note that  for such $k$ and $\ell$, the above formula is equivalent with the binomial identity
\begin{equation}\label{eq3_precise_formula}
{{k+j-1}\choose j}=\sum_{i=0}^j{{\ell+j-i-1}\choose {j-i}}\cdot {{k-\ell+i-1}\choose i}.
\end{equation}

For every positive integer $m$, we have
$$\frac{1}{(1-z)^m}=\sum_{j\geq 0}{{m+j-1}\choose j}z^j.$$
Using the identity
$$\frac{1}{(1-z)^k}=\frac{1}{(1-z)^{\ell}}\cdot \frac{1}{(1-z)^{k-\ell}},$$
we obtain (\ref{eq3_precise_formula}).
\end{proof}

\begin{lemma}\label{formula_alternating_sum}
For every  $j\ge 0$, we have
$$Q_j(x+1) = \sum_{i=0}^ji!{j\choose i}Q_{j-i}(x).$$
\end{lemma}

\begin{proof}
It is of course enough to show that the equality holds whenever we evaluate each side at a positive integer $m$.
The corresponding equality is equivalent with the following binomial identity
\begin{equation}\label{eq_formula_alternating_sum}
{{m+j}\choose j}=\sum_{i=0}^j{{m+j-i-1}\choose {j-i}}.
\end{equation}
The right-hand side of (\ref{eq_formula_alternating_sum}) is the coefficient of 
$t^{m-1}$ in
$$\sum_{p=0}^j(t+1)^{m+p-1}=(t+1)^{m-1}\cdot \frac{(t+1)^{j+1}-1}{t},$$
hence it is equal to the left-hand side of (\ref{eq_formula_alternating_sum}).
\end{proof}

\end{document}

%% file: macros.tex
\usepackage{amsmath,amsfonts,amsthm,mathrsfs}
\usepackage{amssymb}

\usepackage[unicode,bookmarks]{hyperref}
\usepackage[usenames,dvipsnames]{xcolor}
\hypersetup{colorlinks=true,citecolor=NavyBlue,linkcolor=Brown,urlcolor=Orange}

\usepackage[alphabetic,initials]{amsrefs}

\usepackage{enumitem}

\usepackage{chngcntr}

\ifdraft
\usepackage[notcite,notref,color]{showkeys}

\definecolor{labelkey}{gray}{0.5}
\fi

\usepackage{tikz}

\usepackage{tikz-cd}

\usetikzlibrary{matrix,arrows}
\newlength{\myarrowsize} 

\pgfarrowsdeclare{cmto}{cmto}{
	\pgfsetdash{}{0pt} 
	\pgfsetbeveljoin 
	\pgfsetroundcap 
	\setlength{\myarrowsize}{0.6pt}
	\addtolength{\myarrowsize}{.5\pgflinewidth}
	\pgfarrowsleftextend{-4\myarrowsize-.5\pgflinewidth} 
	\pgfarrowsrightextend{.8\pgflinewidth}
}{
	\setlength{\myarrowsize}{0.6pt} 
  	\addtolength{\myarrowsize}{.5\pgflinewidth}  
	\pgfsetlinewidth{0.5\pgflinewidth}
	\pgfsetroundjoin
	\pgfpathmoveto{\pgfpoint{1.5\pgflinewidth}{0}}
	\pgfpatharc{-109}{-170}{4\myarrowsize}
	\pgfpatharc{10}{189}{0.58\pgflinewidth and 0.2\pgflinewidth}
	\pgfpatharc{-170}{-115}{4\myarrowsize+\pgflinewidth}
	\pgfpathclose
	\pgfusepathqfillstroke
	\pgfpathmoveto{\pgfpoint{1.5\pgflinewidth}{0}}
	\pgfpatharc{109}{170}{4\myarrowsize}
	\pgfpatharc{-10}{-189}{0.58\pgflinewidth and 0.2\pgflinewidth}
	\pgfpatharc{170}{115}{4\myarrowsize+\pgflinewidth}
	\pgfpathclose
	\pgfusepathqfillstroke
	\pgfsetlinewidth{2\pgflinewidth}
}

\pgfarrowsdeclare{cmonto}{cmonto}{
	\pgfsetdash{}{0pt} 
	\pgfsetbeveljoin 
	\pgfsetroundcap 
	\setlength{\myarrowsize}{0.6pt}
	\addtolength{\myarrowsize}{.5\pgflinewidth}
	\pgfarrowsleftextend{-4\myarrowsize-.5\pgflinewidth} 
	\pgfarrowsrightextend{.8\pgflinewidth}
}{
	\setlength{\myarrowsize}{0.6pt} 
  	\addtolength{\myarrowsize}{.5\pgflinewidth}  
	\pgfsetlinewidth{0.5\pgflinewidth}
	\pgfsetroundjoin
	\pgfpathmoveto{\pgfpoint{1.5\pgflinewidth}{0}}
	\pgfpatharc{-109}{-170}{4\myarrowsize}
	\pgfpatharc{10}{189}{0.58\pgflinewidth and 0.2\pgflinewidth}
	\pgfpatharc{-170}{-115}{4\myarrowsize+\pgflinewidth}
	\pgfpathclose
	\pgfusepathqfillstroke
	\pgfpathmoveto{\pgfpoint{1.5\pgflinewidth}{0}}
	\pgfpatharc{109}{170}{4\myarrowsize}
	\pgfpatharc{-10}{-189}{0.58\pgflinewidth and 0.2\pgflinewidth}
	\pgfpatharc{170}{115}{4\myarrowsize+\pgflinewidth}
	\pgfpathclose
	\pgfusepathqfillstroke
	\pgfpathmoveto{\pgfpoint{1.5\pgflinewidth-0.3em}{0}}
	\pgfpatharc{-109}{-170}{4\myarrowsize}
	\pgfpatharc{10}{189}{0.58\pgflinewidth and 0.2\pgflinewidth}
	\pgfpatharc{-170}{-115}{4\myarrowsize+\pgflinewidth}
	\pgfpathclose
	\pgfusepathqfillstroke
	\pgfpathmoveto{\pgfpoint{1.5\pgflinewidth-0.3em}{0}}
	\pgfpatharc{109}{170}{4\myarrowsize}
	\pgfpatharc{-10}{-189}{0.58\pgflinewidth and 0.2\pgflinewidth}
	\pgfpatharc{170}{115}{4\myarrowsize+\pgflinewidth}
	\pgfpathclose
	\pgfusepathqfillstroke
	\pgfsetlinewidth{2\pgflinewidth}
}

\pgfarrowsdeclare{cmhook}{cmhook}{
	\pgfsetdash{}{0pt} 
	\pgfsetbeveljoin 
	\pgfsetroundcap 
	\setlength{\myarrowsize}{0.6pt}
	\addtolength{\myarrowsize}{.5\pgflinewidth}
	\pgfarrowsleftextend{-4\myarrowsize-.5\pgflinewidth} 
	\pgfarrowsrightextend{.8\pgflinewidth}
}{
	\setlength{\myarrowsize}{0.6pt} 
  	\addtolength{\myarrowsize}{.5\pgflinewidth}  
 	\pgfsetdash{}{0pt}
	\pgfsetroundcap
	\pgfpathmoveto{\pgfqpoint{0pt}{-4.667\pgflinewidth}}
	\pgfpathcurveto
    {\pgfqpoint{4\pgflinewidth}{-4.667\pgflinewidth}}
    {\pgfqpoint{4\pgflinewidth}{0pt}}
    {\pgfpointorigin}
	\pgfusepathqstroke
}

\newenvironment{diagram*}[2]{%
\[%
\begin{tikzpicture}[>=cmto,baseline=(current bounding box.center),%
	to/.style={->,font=\scriptsize,cap=round},%
	into/.style={cmhook->,font=\scriptsize,cap=round},%
	onto/.style={-cmonto,font=\scriptsize,cap=round},%
	math/.style={matrix of math nodes, row sep=#2, column sep=#1,%
		text height=1.5ex, text depth=0.25ex}]%
}{%
\end{tikzpicture}%
\]%
\ignorespacesafterend%
}

\newcommand{\Dmod}{\mathscr{D}}
\newcommand{\Mmod}{\mathcal{M}}

\newcommand{\ZZ}{\mathbb{Z}}
\newcommand{\QQ}{\mathbb{Q}}

\newcommand{\CC}{\mathbb{C}}

\newcommand{\shf}[1]{\mathscr{#1}}

\def\overbar#1#2#3{{%
	\setbox0=\hbox{\displaystyle{#1}}%
	\dimen0=\wd0
	\advance\dimen0 by -#2 
	\vbox {\nointerlineskip \moveright #3 \vbox{\hrule height 0.3pt width \dimen0}%
		\nointerlineskip \vskip 1.5pt \box0}%
}}

\newcommand{\shO}{\shf{O}}

\makeatletter
\let\@@seccntformat\@seccntformat
\renewcommand*{\@seccntformat}[1]{%
  \expandafter\ifx\csname @seccntformat@#1\endcsname\relax
    \expandafter\@@seccntformat
  \else
    \expandafter
      \csname @seccntformat@#1\expandafter\endcsname
  \fi
    {#1}%
}
\newcommand*{\@seccntformat@subsection}[1]{%
  \textbf{\csname the#1\endcsname.}
}
\makeatother

\makeatletter
\let\@paragraph\paragraph
\renewcommand*{\paragraph}[1]{%
	\vspace{0.3\baselineskip}%
	\@paragraph{\textit{#1}}%
}
\makeatother

\counterwithin{equation}{subsection}
\counterwithout{subsection}{section}
\counterwithin{figure}{subsection}

\newtheorem{theorem}{Theorem} [subsection]
\newtheorem*{theorem*}{Theorem}
\newtheorem{lemma}[theorem]{Lemma}
\newtheorem*{lemma*}{Lemma}
\newtheorem{corollary}[theorem]{Corollary}
\newtheorem{proposition}[theorem]{Proposition}
\newtheorem*{proposition*}{Proposition}

\theoremstyle{definition}
\newtheorem{definition}[theorem]{Definition}
\newtheorem*{definition*}{Definition}
\newtheorem{remark}[theorem]{Remark}

\newtheorem{question}[theorem]{Question}

\newtheorem{example}[theorem]{Example}
\newtheorem*{example*}{Example}
\newtheorem*{problem*}{Problem}

\theoremstyle{plain}

\newcommand{\theoremref}[1]{\hyperref[#1]{Theorem~\ref*{#1}}}
\newcommand{\lemmaref}[1]{\hyperref[#1]{Lemma~\ref*{#1}}}
\newcommand{\definitionref}[1]{\hyperref[#1]{Definition~\ref*{#1}}}
\newcommand{\propositionref}[1]{\hyperref[#1]{Proposition~\ref*{#1}}}
\newcommand{\conjectureref}[1]{\hyperref[#1]{Conjecture~\ref*{#1}}}
\newcommand{\corollaryref}[1]{\hyperref[#1]{Corollary~\ref*{#1}}}
\newcommand{\exampleref}[1]{\hyperref[#1]{Example~\ref*{#1}}}

\makeatletter
\let\old@caption\caption
\renewcommand*{\caption}[1]{%
	\setcounter{figure}{\value{equation}}%
	\stepcounter{equation}%
	\old@caption{#1}\relax%
}
\makeatother

\newcounter{intro}

\newtheorem{intro-conjecture}[intro]{Conjecture}
\newtheorem{intro-corollary}[intro]{Corollary}
\newtheorem{intro-theorem}[intro]{Theorem}

\newcommand{\parref}[1]{\hyperref[#1]{\S\ref*{#1}}}

\makeatletter
\newcommand*\if@single[3]{%
  \setbox0\hbox{${\mathaccent"0362{#1}}^H$}%
  \setbox2\hbox{${\mathaccent"0362{\kern0pt#1}}^H$}%
  \ifdim\ht0=\ht2 #3\else #2\fi
  }
\newcommand*\rel@kern[1]{\kern#1\dimexpr\macc@kerna}
\newcommand*\widebar[1]{\@ifnextchar^{{\wide@bar{#1}{0}}}{\wide@bar{#1}{1}}}
\newcommand*\wide@bar[2]{\if@single{#1}{\wide@bar@{#1}{#2}{1}}{\wide@bar@{#1}{#2}{2}}}
\newcommand*\wide@bar@[3]{%
  \begingroup
  \def\mathaccent##1##2{%
    \if#32 \let\macc@nucleus\first@char \fi
    \setbox\z@\hbox{$\macc@style{\macc@nucleus}_{}$}%
    \setbox\tw@\hbox{$\macc@style{\macc@nucleus}{}_{}$}%
    \dimen@\wd\tw@
    \advance\dimen@-\wd\z@
    \divide\dimen@ 3
    \@tempdima\wd\tw@
    \advance\@tempdima-\scriptspace
    \divide\@tempdima 10
    \advance\dimen@-\@tempdima
    \ifdim\dimen@>\z@ \dimen@0pt\fi
    \rel@kern{0.6}\kern-\dimen@
    \if#31
      \overline{\rel@kern{-0.6}\kern\dimen@\macc@nucleus\rel@kern{0.4}\kern\dimen@}%
      \advance\dimen@0.4\dimexpr\macc@kerna
      \let\final@kern#2%
      \ifdim\dimen@<\z@ \let\final@kern1\fi
      \if\final@kern1 \kern-\dimen@\fi
    \else
      \overline{\rel@kern{-0.6}\kern\dimen@#1}%
    \fi
  }%
  \macc@depth\@ne
  \let\math@bgroup\@empty \let\math@egroup\macc@set@skewchar
  \mathsurround\z@ \frozen@everymath{\mathgroup\macc@group\relax}%
  \macc@set@skewchar\relax
  \let\mathaccentV\macc@nested@a
  \if#31
    \macc@nested@a\relax111{#1}%
  \else
    \def\gobble@till@marker##1\endmarker{}%
    \futurelet\first@char\gobble@till@marker#1\endmarker
    \ifcat\noexpand\first@char A\else
      \def\first@char{}%
    \fi
    \macc@nested@a\relax111{\first@char}%
  \fi
  \endgroup
}
\makeatother